\documentclass[12pt,a4paper]{amsart}
\usepackage{a4wide} 

\usepackage{pdfpages}
 
 \usepackage{amsmath, amssymb, amsfonts, enumerate}
\usepackage[all]{xy}
\usepackage{amscd}
\usepackage{comment}
\usepackage{mathtools}
\usepackage{tikz} 
\usepackage{tikz-cd}

\usepackage{url}
\usepackage{hyperref}

\newtheorem{theoremletter}{Theorem}

\newtheorem{corollaryletter}{Corollary}

\makeatletter
\newtheorem*{rep@theorem}{\rep@title}
\newcommand{\newreptheorem}[2]{%
\newenvironment{rep#1}[1]{%
 \def\rep@title{#2 \ref{##1}}%
 \begin{rep@theorem}}%
 {\end{rep@theorem}}}
\makeatother

\newtheorem{theorem}{Theorem}[section]
\newtheorem*{theorem*}{Theorem}
 \newtheorem*{conjecture*}{Conjecture}
  
  \newtheorem*{corollary*}{Corollary}
  
  \newtheorem{lemma}[theorem]{Lemma}
\newtheorem{corollary}[theorem]{Corollary}
\newtheorem{proposition}[theorem]{Proposition}

\newreptheorem{theorem}{Theorem}
 \newreptheorem{proposition}{Proposition}
 \newreptheorem{corollary}{Corollary} 

 \theoremstyle{definition}
 \newtheorem{definition}[theorem]{Definition} 

 \newtheorem{remark}[theorem]{Remark}

 \newtheorem{example}[theorem]{Example}

\numberwithin{equation}{section}
 
\newcommand {\N}{\mathbb{N}} 
\newcommand {\Z}{\mathbb{Z}} 
 
\newcommand {\Q}{\mathbb{Q}} 
\newcommand {\C}{\mathbb{C}}

\newcommand{\CC}{\mathcal{C}}
\newcommand{\DD}{\mathcal{D}}
\newcommand{\EE}{\mathcal{E}}

\newcommand{\LL}{\mathcal{L}}

\newcommand{\OO}{\mathcal{O}}

\newcommand{\VV}{\mathcal{V}}

\newcommand{\XX}{\mathcal{X}}
\newcommand{\YY}{\mathcal{Y}}

\newcommand{\Proj}{\mathbb{P}}

\DeclareMathOperator{\Ker}{Ker}

\DeclareMathOperator{\Id}{Id}

\DeclareMathOperator{\Aut}{Aut}

\DeclareMathOperator{\rank}{rank}
\DeclareMathOperator{\NS}{NS}

\DeclareMathOperator{\Pic}{Pic}

\usepackage{setspace}
\begin{document}	
\title[Parshin-Arakelov theorem and  integral points on elliptic curves]
{On Parshin-Arakelov theorem and uniformity of $S$-integral sections on elliptic surfaces}
\author[Xuan-Kien Phung]{Xuan Kien Phung}
\address{Universit\'e de Strasbourg, CNRS, IRMA UMR 7501, F-67000 Strasbourg, France}
\email{phung@math.unistra.fr}
\subjclass[2010]{}

\begin{abstract}   
Let $f \colon X \to B$ be a complex elliptic surface and let $\DD \subset X$ 
be an integral divisor dominating $B$.    
It is well-known that the Parshin-Arakelov theorem implies 
the Mordell conjecture over complex function fields by 
a beautiful covering trick of Parshin.  
In this article, we construct a similar map in the context of $(S, \DD)$-integral points on  
elliptic curves over function fields to obtain a new proof 
of certain uniform finiteness results on the number 
of $(S, \DD)$-integral points. 
A second new proof is also given by establishing a uniform 
bound on the canonical height   by means of the tautological inequality. 
In particular, our construction provides certain uniform quantitative informations 
on the set-theoretic intersection of curves with the singular divisor 
in the compact moduli space of stable curves. 
\end{abstract}

\maketitle
 
\setcounter{tocdepth}{1}
\tableofcontents


\section{Introduction}
It is well-known that in the case of function fields in characteristic zero, 
the Parshin-Arakelov theorem (cf. Theorem \ref{t:parshin-arakelov}) implies the Mordell conjecture (cf. \cite{parshin-68}). 
Moreover, by establishing a uniform version of the Parshin-Arakelov theorem (cf. Theorem \ref{t:caporaso}), Caporaso obtained a certain uniform version of the Mordell conjecture over function fields. We explain this in some more details below. 
\par
\subsection*{Notations} 
The symbol $\sim$ stands for the linear equivalence relation of divisors. 
The  cardinality of a set $A$ is denoted by $\# A$. 
We fix throughout an irreducible smooth projective complex curve $B$ of genus $g$. 
Let $K= \C(B)$ be the function field of $B$. 
A family of curves $\YY \to B$ (resp. a curve $Y/K$) is \emph{isotrivial}  
if the induced moduli (resp. rational) map to the moduli space of curves is constant, 
or equivalently, if the Kodaira-Spencer class of the family $\YY/B$ (resp. of $Y/K$) 
is zero (cf. \cite[Chapter 3]{gasbarri}). 
\par
For a nonisotrivial minimal surface $X \to B$, we   call the \emph{type} of $X$ 
the finite subset $S \subset B$ above which the fibres of $X$ are not smooth. 
Let $F_q(B,S)$ be the set of non-isotrivial minimal surfaces over $B$ of type $S$ and whose general fibres have genus $q$. 
When $q \geq 2$,   we have the 
following celebrated result of Parshin-Arakelov:  

\begin{theorem}
[Parshin-Arakelov]
\label{t:parshin-arakelov} 
The set $F_q(B,S)$ is finite for every $q \geq 2$. 
\end{theorem}

\begin{proof}
See \cite{parshin-68} in the case $S= \varnothing$ and \cite{arakelov-71} for the general case. 
\end{proof} 

The above  theorem was latter reinforced by Caporaso as follows (\cite[Theorem 3.1]{capo-99}). 
 
\begin{theorem}
[Caporaso]
\label{t:caporaso}
There exists a function $C \colon \N^3 \to \N$ such that for all $q \geq 2$,  
\[
\#F_q(B,S)\leq C(q,g,s). 
\]
 \end{theorem}
 
 Let $(X \to B) \in F_q(B,S)$ where $q \geq 2$. In \cite{parshin-68}, Parshin constructed a map $\alpha_X$ which associates 
to each section 
$\sigma \in  X(B)$ a nonisotrivial families of curves of bounded genus $Y_\sigma \to B$ 
which factors through $X \to B$ and such that 
$\sigma(B)$ is exactly the branch locus of the cover $Y_\sigma \to X$. 
The obtained nonisotrivial families of curves $Y_\sigma$ are all defined over a 
finite number of curves $B'$ and the possible types $S' \subset B'$ are also finite. 
The key property of the map $\alpha_X$ is that all its fibres are finite. This is direct implication 
of the property that $\sigma$ is uniquely determined as the branch locus of the cover $Y_\sigma \to X$ 
and that the number of possible covers $Y_\sigma \to X$ is finite by the de Franchis theorem since 
$q \geq 2$.  
Therefore, the Parshin-Arakelov theorem implies the Mordell conjecture over function fields, i.e., 
$X(B)=X_K(K)$ is finite. 
With additional arguments proving that $B', S'$ can be taken in a uniformly bounded family and 
the cardinalities of the fibres of $\alpha_X$ are uniformly bounded, 
Theorem \ref{t:caporaso} of Caporaso implies the following  uniform version of the Mordell conjecture. 

\begin{theorem}
[Caporaso] 
\label{t:mordell-uniform-caporaso} 
There exists a function $M \colon \N^3 \to \N$ such that 
for every $q \geq 2$ and $X \in F_q(B,S)$, the set $X_K(K)=X(B)$ contains at most $M(q,g,s)$ rational points. 
\end{theorem}

\begin{proof}
See \cite[Theorem 4.2]{capo-99}. 
\end{proof}

\subsection{$(S, \DD)$-integral points}

The main objects studied in this article is 
the following. 

\begin{definition} 
[$(S, \DD)$-integral point and section] 
\label{d:s-d-integral-point-geometric}
 Let $f \colon  X \to B$ be a proper flat morphism of integral varieties. 
Let $S \subset B$ be a subset 
and let $\DD \subset X$ be an effective divisor. 
A section $\sigma \colon B \to X$ is $(S, \DD)$\emph{-integral}
 if it satisfies the set-theoretic condition: 
\begin{equation}
\label{e:s-d-integral-point-geometric}
f(\sigma(B) \cap \DD) \subset S. 
\end{equation}
For every $P \in X_K(K)$, let $\sigma_P \colon B \to X$ 
be  the induced section. 
Then $P$ is said to be $(S, \DD)$\emph{-integral} 
if the section $\sigma_P$ is $(S, \DD)$-integral. 
\end{definition} 

Definition \ref{d:s-d-integral-point-geometric} is the geometric interpretation  
of integral solutions of Diophantine equations. 
A more  general definition of integral points is proposed and studied 
in \cite{phung-19-abelian}.  

\subsection{Parshin's map $\alpha_X$ and integral points}

The first goal of the article is to construct a map analogous 
to Parshin's map $\alpha_X$ to obtain a new proof for known uniform finiteness results 
(cf. Corollary \ref{c:height-iso-elliptic-silverman}) 
on integral points on a nonisotrivial elliptic curve.  
Let $f \colon X\to B$ be a nonisotrivial elliptic surface of type $T$ of cardinal $t$.  
Let $\DD$ be an effective reduced horizontal divisor in $X$ 
and $S \subset B$ of cardinal $s\in \N$. 
We show that 
(cf. Section \ref{section:uniform-caporaso-elliptic}): 

\begin{theoremletter} 
\label{t:uniform-caporaso-elliptic-1}
The set of $(S, \DD)$-integral sections of $X$ is finite and 
uniformly bounded by a function depending only on $g$, $s$, $t$, 
$\deg \DD_K$, the number of 
ramified points in the cover $\DD \to B$, 
and the number of singular points on $\DD$.     
\end{theoremletter}

When $\DD$ is a section, $\deg \DD_K=1$ and $\#\DD_{ram} \cup \DD_{sing}= \varnothing$. 
We recover in particular the uniform result in \cite{hindry-silverman-88} whose proof uses height theory: 

\begin{corollary}
Let $(O)$ be the zero section of $X$. 
The set of $(S, \DD)$-integral sections is uniformly 
bounded by a function depending only on $g$, $s$, $t$. 
\end{corollary}

The main idea of our construction is similar to the method of Parshin in the following way. 
For each nonisotrivial elliptic surface $X \to B$ with a finite subset $S \subset B$ and 
a horizontal effective divisor $\DD$, 
we define a map $\beta_{X, \DD, S}$ which associates to each $(S, \DD)$-integral section $\sigma \in X(B)$ 
a cover $Y_\sigma \to X$ whose (horizontal part of the) branch locus is $\sigma(B) \cup \DD$. 
The induced maps $Y_\sigma \to B$ are nonisotrivial families of curves of uniformly  
bounded genus $\geq 2$ and of uniformly bounded type. 
\par
The new key point is that the extra presence of $\DD$, and not just the section $\sigma(B)$,  
in the branch locus turns out to be exactly what we need, altogether with the de Franchis theorem 
on elliptic subfields due to Tamme-Kani (cf. Theorem \ref{c:kani}),  
to show that the fibres of the map $\beta_{X, \DD, S}$ are uniformly bounded. 
Thereby, we obtain from Theorem \ref{t:mordell-uniform-caporaso} 
a uniform result of Siegel theorem in the case of function fields 
with a new method other than the classical methods using heights 
(cf. Corollary \ref{c:height-iso-elliptic-silverman}). 

\subsection{A negative result on the Parshin-Arakelov theorem}
It is   natural to expect that the above method can be applied to obtain  
finiteness results on integral points of bounded denominators 
on elliptic curves over function fields (as in Corollary \label{c:height-iso-elliptic-silverman} below). 
For this to be done using the construction of the map $\beta_{X, \DD, S}$, 
it turns out that we need the property saying that the union 
$\cup_{S \subset B,  \#S \leq s} F_q(B,S)$ is finite for every $s \in \N$. 
Unfortunately, the second goal of this chapter is to establish 
a  uniform negative finiteness result on the union $\cup_{S \subset B,  \#S \leq s} F_q(B,S)$ 
(cf. Section \ref{section:negative-parshin-arakelov-1}): 

\begin{theoremletter} 
\label{t:negative-parshin-arakelov}
 For all large enough $q$ and $s$ depending only on the genus $g$ of $B$, the union 
\[
\cup_{S \subset B , \# S \leq s} F_q(B,S) 
\]
is uncountably infinite. 
Moreover, there exists $N(q,s,g) \in \N$, 
a Zariski dense open subset $I \subset \Proj^{N}$, and 
a map $I \to \cup_{S \subset B , \# S \leq s} F_q(B,S)$ with uniformly bounded finite fibres. 
 \end{theoremletter}

The above result shows that we cannot extend directly our method 
to recover the finiteness of 
integral points of bounded denominators 
(cf. Corollary \ref{c:height-iso-elliptic-silverman}). 
\par
From Theorem \ref{t:negative-parshin-arakelov}, 
we obtain the following 
geometric information on the compact fine moduli spaces $\mathcal{M}_{q, n}$ of stable curves of genus $q$ 
with level $n\geq 3$-structure. 
Let $\Delta \subset \mathcal{M}_{q,n}$ be the divisor locus of 
singular curves.  

\begin{corollaryletter} 
For large enough $q, s \in \N$ depending only on $g$, 
there exists uncountably many nonconstant morphisms $h \colon B \to \mathcal{M}_{q, n}$, 
up to automorphisms of $B$, 
such that the set-theoretic intersection 
$h(B) \cap \Delta$ has no more than $s$ points. 
\end{corollaryletter}

We shall apply our strategy to prove a certain uniform finiteness result 
on unit equations over function fields in Section \ref{s:parshin-arakelov-unit-equation} 
(cf. Theorem \ref{t:unit-general-3}). 
The result obtained is nontrivial but far from being optimal 
(cf. Remark \ref{r:application-parshin-arakelov-to-unit-evertse}). But again, 
as in the case of integral points on   elliptic surfaces, 
our proof is new in the sense that it does not reduce to 
establish any height bound on the set of solutions as in traditional approaches 
in the literature.   

\subsection{Canonical height bound revisited and tautological inequality}
Recall a well-known effective 
bound of the canonical height $\widehat{h}$ of integral 
points on elliptic surfaces (cf. \cite[Corollary 8.5]{hindry-silverman-88}): 
 
\begin{theorem} 
[Hindry-Silverman]
\label{t:height-iso-elliptic}
Let $X \to B$ be a minimal elliptic surface with a section $(O)$. 
 Then for any $(S,(O))$-integral point $P \in X_K(K)$, we have 
 $ \widehat{h}(P) \leq 25 \chi(X) + 6 g + 2 s$, 
where $\chi(X)$ is the Euler-Poincar\' e characteristic of $X$ and $s=\#S$.  
\end{theorem}

For traceless families of abelian varieties, a bound on the canonical height of $(S, \DD)$-integral points 
which is linear in terms of $s= \#S$ is known in \cite{buium-94}. 
Such results are unknown for a general family of abelian varieties. 
From Theorem \ref{t:height-iso-elliptic}, we can   obtain the 
following effective finiteness result on integral points of \emph{bounded denominators}.

\begin{corollary}
\label{c:height-iso-elliptic-silverman}
Let $f \colon X \to B$ be a nonisotrivial elliptic surface of type $T \subset B$ with a zero section $(O)$. 
There exists $\alpha, \beta, \gamma >0$ depending only on $g$ and $t= \# T$ such that 
for every $s \in \N$,  
the   union: 
\begin{align*}
I_s \coloneqq \cup_{S \subset B, \#S  \leq s} 
\{(S, (O))\text{-integral points of } X_K(K) \} \subset X_K(K)
\end{align*}
is finite and $\#I_s \leq (\alpha s + \beta)^\gamma$. 
Moreover, the same result holds when 
$(O)$ is replaced by any horizontal integral divisor $\DD \subset X$ 
but with $\alpha, \beta, \gamma$ depending also on the genus of $\DD$. 
\end{corollary} 

The above  result is   remarked without proof  in \cite[2.9]{shioda-schutt-lecture}. 
A stronger result concerning generalized integral points is given in \cite{phung-19-elliptic}. 
When $X \to B$ is a  constant family of abelian varieties, certain finiteness uniform results on the height 
and on the integral points 
are obtained in \cite{noguchi-winkelmann-04} and \cite{phung-19-uniform-noguchi}. 

\begin{proof}[Proof of Corollary \ref{c:height-iso-elliptic-silverman}]
As the   bound in Theorem \ref{t:height-iso-elliptic} depends only on $\#S$, 
the exact same proof of \cite[Theorem 8.1]{hindry-silverman-88} can be applied to give the result. 
The constants $\alpha, \beta$ can be given as a function of  $g$ 
and that $\gamma$ is half of the Mordell-Weil rank of $X_K(K)$, 
which is bounded by $2(2g-2+t)$ by the Shioda-Tate formula (cf. \cite[Theorem 2.5]{shioda-elliptic-modular-surface}).   
\par
Now let $\DD \subset X$ be an integral curve which is finite over $B$. 
Let $C \to \DD$ be the normalization morphism 
 and let $h \colon C \to B$ be the induced finite morphism of degree $d$. 
For each finite subset $S \subset B$ of cardinal $\#S\leq s$, 
let $S' = f'^{-1}(S) \subset C$ then  
$ \# S' \leq s' \coloneqq ds$. 
Consider the elliptic surface $f' \colon  X'= X \times_B C \to C$ which is also non-isotrivial 
(cf. Theorem \ref{t:isotrivial-dominant}). 
Let $T'$ be the type of $X'$ then $\# T' \leq d t$. 
It is clear that $\DD \times_B C$ splits into sections of $f'$. 
Let $R$ be one of the these sections and let $K'=\C(C)$.  
It follows that we have 
$$
I_s \subset I'_{s'} \coloneqq \cup_{S' \subset B, \#S'  \leq s'} 
\{(S', R)\text{-integral points of } X'(K') \} \subset X'(K'). 
$$
The desired properties of $I_s$ are then obtained from those of $I'_{s'}$. 
Remark   that the constants $\alpha, \beta, \gamma$ now  
depend also on the genus of the divisor $\DD$.  
\end{proof} 

The third goal of the article is to give a new proof of 
the uniform consequence on the canonical height of $(S, \DD)$-integral 
points (as in Theorem \ref{t:height-iso-elliptic}) by using only the tautological inequality 
(cf. Definition \ref{d:tautological-inequality}). 
Thereby, we also give a new proof of Corollary \ref{c:height-iso-elliptic-silverman}. 
\par
For simplicity, we restrict to semistable families of \emph{relative maximal variation}  
and \emph{adaptive} families of effective ample divisors 
(cf. Section \ref{s:tautological}).  

\begin{theoremletter}
\label{t:general-theorem-shorst-list-1}
Let $\XX \xrightarrow{f} \CC \to Z$ be a relative maximal variation family of    semistable 
elliptic surfaces with a zero section $O \colon \CC \to \XX$. 
Let $\DD \subset \XX$ be an adaptive family of ample effective divisors. 
 There exists   $c_1,c_2 >0$ such that 
for every $z \in Z$ and every $P \in \XX_z(k(\CC_z))\setminus \DD_z$,  
 \begin{align}
\label{e:tautological-uniform-bound-height-main}
\widehat{h}_{O_z} (P) \leq c_1s + c_2, \quad \text{ where } s= \# \sigma_P(\CC_z) \cap \DD_z.  
\end{align}
Here, $\widehat{h}_{O_z}$ is the N\' eron-Tate height on $\XX_z(k(\CC_z))$ 
associated to the origin $O_z$ and $\sigma_P \in \XX_z(\CC_z)$ is the corresponding section of $P$.  
\end{theoremletter} 

When $\DD=(O)$ is the image of the zero section, Theorem \ref{t:c-12-independent} shows that 
the constants $c_1, c_2$ in Theorem \ref{t:general-theorem-shorst-list-1}
depend only on the topological invariants $\chi=\chi(\OO_{X_z})$ 
and $g=g(\CC_z)$ (where $z \in Z$) exactly as in Theorem \ref{t:height-iso-elliptic}.

\section{Preliminaries}
Let $X$ be a connected scheme. 
We fix a geometric point $\bar{x}$ to obtain an \' etale fundamental group $\pi_1^{et}(X)=\pi_1^{et}(X, \bar{x})$. 
An \' etale degre $d$ cover of $X$ is equivalent to a $\pi_1^{et}(X)$-set of cardinal $d$, i.e., a continuous homomorphism of topological groups  
\[
\pi_1^{et}(X) \to S_d, 
\]
where $S_d$ denotes the discret group of permutations of a set of $d$ elements.  
 Hence, the number of degree $d$ \'etale covers of $X$ is given by $\#\text{Hom}(\pi_1^{et}(X), S_d)$. 
If $\pi_1^{et}(X)$ is finitely generated with $m$ generators then  
$\#\text{Hom}(\pi_1^{et}(X), S_d) \leq (d!)^m$. 
Moreover, for a regular connected variety $X$ over $\C$, 
we have $\text{Hom}(\pi_1^{et}(X), S_d)=\text{Hom}(\pi_1(X(\C)), S_d)$. 
 If $X=B \backslash T$, where  $T$ is a finite subset of cardinal $t$, 
then the topological fundamental group  $\pi_1(X_{\C}(\C))$ 
is a free group $F_{2g+t-1}$ 
of rank $2g-1+t$.  
Therefore, we see that
\[
\pi_1^{et}(X)=\hat{F}_{2g-1+t}
\]
 is the completion of the free group of $2g-1+t$ generators. 
 In this case, the number of degree $d$ \' etale covers of $B\backslash T$ is bounded above by 
\[
N(d,g,t)=(d!)^{2g+t-1}. 
\] 

 Now return to the situation when $f \colon X \to B$ is a minimal elliptic surface with zero section $O$ and 
$T \subset B$ the type of $X$. 
Let $Y=X\backslash f^{-1}(T)$ then $Y$ is an abelian scheme over $B\setminus T$. 
The multiplication-by-2 morphism $[2]$ on $Y$ is finite \' etale of degree $4$. 
Moreover, its kernel $\Ker [2] \subset Y $, obtained by the base change associated to the zero section,  
is also finite \' etale of degree $4$ over $B\setminus  T$. 
\[
\label{d-parshin-integral-point-diagram-etale}
\begin{tikzcd}
 \Ker [2]  \arrow[d, hook, swap]  \arrow[r]  &  B \setminus T  \arrow[d, "\sigma_O", hook] \\ 
Y \arrow[r,"\text{[2]}"]  & Y.  
\end{tikzcd}
\]
 Similarly, by composing with the translation $\tau_P\in \text{Aut}_B(X)$ for $P\in E(K)$, we see that $\Ker (\tau_P \circ [2])$ is also finite \' etale of degree $4$ over $B \backslash T$. 
Every point $Q \in E(\bar{K})$ such that $[2]Q=P$ is an element of $\Ker [2]$.  
Therefore, 
by composing all the finite \' etale covers of $B \setminus T$ of degree at most 4, 
we obtain a finite cover $B' \to B$ of degree at most $4^{N(4,g,t)}$ which is  ramified only over $T$ and that $2E(K')\supset E(K)$ where $K'=k(B')$. 
 In particular,  by the Riemann-Hurwitz formula, the genus $g'$ of $B'$ is bounded by a function in $g,t$ given by
\[
g'(B) \leq 4^{N(4,g,t)}(g+t-1)+1. 
\]
The same properties hold if we replace $[2]$ by $[d]$ where $d \in \N$. 
We reformulate the above discussion as follows:

\begin{lemma}
\label{p:2-division}
 Let $d \in \N$ and let  $f \colon X\to B$ be an elliptic surface with section  
 of type $T\subset B$. 
Let $E/K$ be the associated elliptic curve.  
There exists a finite cover $h \colon B' \to B$ of projective smooth curves ramified only over $T$   such that $dE(K')\supset E(K)$ where $K'=\C(B')$. 
Moreover, $h$ is of degree at most $(d^2)^{(d!)^{2g+t-1}}$ where $t= \#T$ and the genus $g'$ of $B'$ is bounded by: 
 \[
g' \leq (d^2)^{(d!)^{2g+t-1}}(g+t-1)+1. 
\]
\end{lemma}

The next theorem is standard and valid 
in arbitrary dimension (cf. \cite[Lemma 3.30]{gasbarri}).  

\begin{theorem} 
\label{t:isotrivial-dominant} 
Let $X' /K $ be a smooth projective isotrivial curve. 
Suppose that $X' \to X $ is a dominant $K$-morphism of smooth projective curves. 
Then $X$ is also isotrivial over $K$. 
\end{theorem}
 
However, using only the de Franchis theorem, 
we can give a short proof of Theorem \ref{t:isotrivial-dominant}. 
Since we do not know if such a proof exists already in the literature, we 
include it here.  

\begin{proof}[Proof of Theorem \ref{t:isotrivial-dominant}] 
Consider any models $\mathcal{X}'\to B$ and $\mathcal{X}\to B$ of $X'$ and $X$ respectively.  
By the resolution of indeterminacy of surfaces, we can find a dominant $B$-morphism $\mathcal{X}' \to \mathcal{X}$ commuting with $\mathcal{X}\to B$ and $\mathcal{X}' \to B$. 
\par
Suppose   on the contrary that $X$ is nonisotrivial but $X'$ is isotrivial.  
Let $C$ be a general fibre of $\mathcal{X}$.   
As $\mathcal{X}'$ is isotrivial, its general curves over $B$ are all isomorphic to each other which we can then denote by a single curve $C'$. 
As $X$ is nonisotrivial, the dominating $B$-morphism $\mathcal{X}' \to \mathcal{X}$ 
induces infinitely many pairwise non isomorphic curves which are fibres of 
$\mathcal{X}$ and which are dominated by $C'$. Note that these dominating maps are all of the same degree $d= [k(X')\colon k(X)]$. 
Let $g',g$ be respectively the genus of $X'$ and $X$. 
If $g\geq 2$ then $g'\geq 2$ by the Riemann-Hurwitz formula. 
However, the de Franchis theorem (cf.  \cite[Theorem XXI.8.27]{Arbarello-II}) says that up to isomorphisms, there are only finitely many   curves $D$ of genus $g'$ such that there is a dominant morphism $C' \to D$. 
Therefore, we obtain a contradiction. 
Similarly, \cite[Theorem XXI.8.27]{Arbarello-II} implies a contradiction when $g=1$. 
The case $g=0$ cannot occur since otherwise $X$ would be the trivial projective line. 
\end{proof}

\begin{lemma}
\label{l:integral-point-extreme-case}
Let $L\sim 0$ be a line bundle on the elliptic surface $f \colon X \to B$ with a section $(O)$. 
Then  for each vertical divisor $V$ on $X$, the linear system $|(O)+V|$ 
consists of effective divisors of the form $(O)+F$ where $F\sim V$.   
\end{lemma}

\begin{proof}
Indeed, let $\DD \in H^0(X,(O)+V)$ then $\DD \sim (O)+V$. 
 We write $\DD=H+F$ where $H$ is an effective horizontal divisor and $F$ is vertical. 
Then the degree of $H$ on the generic fibre is the same as $(O)$ which is $1$. 
This means that $H$ is a section of $X \to B$ corresponding to a rational point $P \in X_K(K)$. 
Hence, the image of $H - (O)$ in $NS(X)/T(X)$ is zero. 
Let $T(X) \subset NS(X)$ be the subgroup generated by the zero section and all vertical divisors. 
We deduce by the Shioda-Tate isomorphism (cf. \cite[Theorem 1.3, Corollary 5.3]{shioda-tate-formula} 
or \cite{shioda-schutt-lecture}) $E(K) \simeq NS(X)/T(X)$ that $H=(O)$. 
 Thus $\DD=(O)+F$ and $F\sim V$ as claimed.  
 \end{proof}
\begin{lemma}
\label{l:auto-elliptic}
Let $E/K$ be an elliptic curve over a field $K$. 
Let $D$ be a divisor of degree $d\geq 1$ on $E$. 
 Let $\iota=-\Id$ be the involution. Suppose that $j(E)\neq 0,1728$ then we have: 
\begin{enumerate} [\rm (i)]
 \item
 $\mathrm{Isom}_K(E)=E(K) \rtimes \{\pm \Id\}$. 
  For all $P\in E(K)$, we have $\tau_P \circ \iota =\iota \circ \tau_{-P}$ where $\tau_P$ is the translation-by-$P$ map. 
 \item  
 There exists at most one point $R\in E(K)$ modulo the $d$-torsion group $E[d]$    
such that all isomorphisms $u\in \text{Isom}_K(E)$ verifying $u(D) \sim D$ 
are of the form $\tau_{P}$ or $\iota \circ \tau_{R+P}$ where $P$ is a rational $d$-torsion point of $E(K)$.  
\end{enumerate}
\end{lemma}

\begin{proof}
For (i) see Theorem \ref{t:translation}.(4) and remark that $\Aut(E)=\{\pm \tau\}$ since 
$j(E)\neq 0,1728$. The second statement is easily checked. 
For (ii), let $u\in \text{Isom}_K(E)$ be such that $u(D)\sim D$. 
From $(i)$, we know that $u$ is of the form $\pm \Id \circ \tau_P$ for some $P\in E(K)$. 
Consider the $K$-divisor $D-d [O]$ on $E$. 
We have $\deg(D-d[O])=0$. 
Hence the isomorphism of groups 
\[
E(K)\to \textbf{Pic}^0_E(K) ,\quad Q\mapsto \mathcal{L}([Q]-[O])  \quad \text{(cf. \cite[Chapter 1]{milne-JVs})}
\] 
implies that $D-d[O]\sim [Q]-[O]$ for some rational point $Q\in E(K)$.  
Thus, we find that $D\sim (d-1)[O]+[Q]$. 
Suppose first that $u=\tau_P$ for some $P\in E(K)$. 
Then $D\sim u(D)\sim (d-1)[P]+[P+Q]\sim (d-1)[P] + [P]  + [Q] - [O]$. 
We deduce that $d[O]\sim d[P]$, i.e., $d([P]-[O])\sim 0$. 
This means exactly that $[d]P=O$, i.e., $P$ is a rational $d$-torsion point of $E$. 
Suppose now that $u= (-\Id)\circ \tau_P$ for some $P\in E(K)$. 
Similarly, we find that 
\[
u(D)\sim (d-1)[-P]+[-P-Q]\sim (d-1)(2[O]-[P]) + (2[O]-[P]) + (2[O]-[Q]) -[O].
\] 
As $D \sim u(D)$, we deduce that $(d+2)[O] \sim d[P] + 2[Q]$. 
Hence $d([P]-[O])\sim 2([Q]-[O])$ and 
thus $[d]P=[2]Q \in E(K)$. 
It suffices to let $R$   a $d$-division point of $[2]Q$ to conclude. 
\end{proof}

\section{Parshin-Arakelov theorem and uniform finiteness of integral points}
\label{section:uniform-caporaso-elliptic}

 We are now in position to construct the map $\beta_{X, \DD, S}$ 
defining on the set of $(S, \DD)$-integral points 
of a nonisotrvial elliptic surface $X \to B$ 
to recover the following uniform result in Theorem \ref{t:uniform-caporaso-elliptic-1}  
without establishing any height bound for integral points. 
 
\begin{proof}[Proof of Theorem \ref{t:uniform-caporaso-elliptic-1}] 
We   denote by $(P)= \sigma_P(B) \subset X$ 
the image section associated to a rational point $P \in X_K(K)$. 
\par
Let $D=\DD_K$ be the pullback of $\DD$ to the generic fibre. 
Let $d= \deg D$ and let $O \in E(K)$ be the zero element. 
We can clearly assume that $\DD$ is integral thus contains no vertical components. 
By Lemma \ref{p:2-division}, 
 every rational point in $E(K)$ has a $(d+1)$-th root in a base extension $K'/K$ of degree at most $16^{2g+t-1}$ which is ramified only above $T$ and with genus $g(K')$ bounded in terms of $g,t$.  
Therefore, up to making the corresponding base change $K' /K$, 
 all $(S, \DD)$-integral points belong to $(d+1)E(K)$.  
 \par
Consider the zero-degree divisor $D-d[O]$. Since $E(K)\neq \emptyset$, we have 
\[
\textbf{\textit{P}}^0_E(K)\simeq E(K) \quad \text{(cf. \cite[Chapter 1]{milne-JVs})
}
\]
where $\textbf{\textit{P}}^0_E(-)$ is the functor which associates to each $K$-scheme $V$ 
the group of families of invertible sheaves on $E$ of degree $0$ 
parametrized by $V$, modulo the trivial family. 
In particular, $\textbf{\textit{P}}^0_E(K)$ is the group of equivalence classes of degree $0$ invertible sheaves on $E$.  
We deduce that $D-d[O] \sim [Q]-[O]$ for some $Q\in E(K)$. 
Now let $P \in E(K)$ be an $(S, \DD)$-integral point. We can suppose $P \in (d+1)E(K)$ by the first paragraph. 
Denote by $R_P, R_Q\in E(K)$ certain $(d+1)$-th roots of $P$ and $Q$ respectively. 
In particular,  
\[
[P]+D\sim (-d[O]+(d+1)[R_P] ) + (d-1)[O]-d[O]+(d+1)[R_Q] \sim (d+1)([R_P]+[R_Q] - [O])
\]
Since $\C(X)=K(E)$, we can extend linear equivalence relations on $E$ to linear equivalence relations on $X$ modulo vertical divisors. 
Thus, for some vertical divisor $F$ on $X$, we have:  
\[
(P) +\DD\sim (d+1)((R_P)+(R_Q)-(O))+ F
\]  
Remark that $f^*(\Pic B) \subset \Pic X$ and   every fibre $X_b$ is integral whenever $b \in B \setminus T$. 
Since $\text{Pic}^0(B)$ is a divisible group, we can write $F\sim (d+1)W-W_T -\varepsilon f^*[b_0]$ for some vertical divisor $W$, $W_T$ and some fixed point $b_0 \in T$ with $0 \leq \varepsilon \leq d$ such that $W_S$ is effective and $f(W_T)\subset T$. 
Therefore, for $L_P= (R_P)+(R_Q)-(O)+W$, we find that 
\[
Z=(P)+\DD+W_T+\varepsilon f^*[b_0] \sim (d+1)L_P. 
\]
\par
We now determine a minimal nonisotrivial fibration of curves associated to $P$ using the tool of cyclic covers. 
By Proposition \ref{p:cyclic-covers}, we obtain  
a degree $d+1$ simple cyclic cover $X'\to X$ of surfaces associated to the data $Z\sim (d+1)L_P$. 
Let $Y \to X'$ be the strict resolution of singularity of $X'$. Consider the
composition $f_P\colon Y \to X' \xrightarrow{} X \to B$. 
Then by the Riemann-Hurwitz formula calculated on general fibres $Y_b \to X_b$, we find that  $f_P$ is a family of curves of genus 
$q$ satisfying $2q-2= d(d+1)$ thus 
\begin{equation}
\label{e:q-parshin-arakelov-integral}
q=(d^2+d+2)/2.
\end{equation}

Let $Z_{sm}$ be the set of regular points of $Z$ and $\DD_{ram} \subset \DD$ be the set of 
ramified points of the finite cover $\DD \to B$. 
Then $W= f(Z_m \cup \DD_{ram}) \subset B$ is finite. 
Consider $b \in B \setminus (W \cup T)$. Then 
$X_b$ is a smooth curve and $Z_b\subset f^{-1}(b)$ is a smooth divisor since $b \notin W\cup T$. 
By Proposition \ref{p:cyclic-covers-fonctorial}, 
$Y_b= f_P^{-1}(b) \to X_b$ is also the degree $d+1$ simple cyclic cover 
associated to the data $Z_b \sim (d+1)(L_P)_b$. 
Hence, Proposition \ref{p:cyclic-covers} implies that the 
fibre $Y_b= f_P^{-1}(b)$ is smooth whenever $b \in B \setminus (W \cup T)$. 
By Theorem \ref{t:isotrivial-dominant}, $Y\to B$ is nonisotrivial since $X$ is nonisotrivial. 
Contract any rational curves on $Y$ if necessary, we obtain a minimal nonisotrivial family over $B$ of type $W \cup T$ of curves of genus $q$. 
Thus $Y \in F_q(B,W \cup T)$. 
\par
Denote $\DD_{sm}$ the set of regular points of $\DD$. Then  
$\DD_{sm} \subset Z_{sm} \cup f^{-1}(S \cup T)$ since the section $(P)$ is smooth 
and does not intersect $\DD$ at fibres lying over $B \setminus S$ by definition of $(S, \DD)$-integral points. 
Hence, $Y \in F_q(B, S')$ where 
$S'=W \cup S \cup  T$ is independent of the choice of the $(S, \DD)$-integral point $P$. 
Let $\DD_{sing}= \DD \setminus \DD_{sm}$. Note that 
\begin{equation}
\label{e:s'-parshin-arakelov-integral}
\#S' \leq s+t+ \# \DD_{ram}+ \#f( \DD_{sing}). 
\end{equation}
If $\DD$ is the zero section then $S'= S \cup T$ and thus $\#S' \leq s+t$.  
 Therefore, we have constructed a Parshin-type map denoted $\beta$ and given by
\begin{align*}
\{(S,\DD)\text{-integral points}\} &\xrightarrow{\,\,\, \beta \,\,\,}  F_q(B,S')  \\
P &\longmapsto (f_P \colon Y \to X \to B)
\end{align*}
with the property that the horizontal part of the branch locus of $Y \to X$ is 
$(P) \cup \DD$. 
\par
We claim that the fibres of the map $\beta$ is uniformly bounded. 
Indeed, let $Y\to B$ be an element of $F_q(B,S)$ that belongs to the image of $\beta$. 
Theorem \ref{c:kani} implies that 
there is at most $M(q,d+1)$ possible $(d+1)$-covers $Y\to X$ up to composition with an element of $\Aut_B(X)=\Aut_K(E)$. 
Here, the function  $M(q,d+1)$ is given in Theorem \ref{c:kani}. 
Hence, it suffices to prove that each such class of $(d+1)$-covers $Y \to X$ (modulo $\Aut_B(X)$) is the image of at most $4d^2$ possible $(S,\DD)$-integral points. 
Indeed, let $P, P'\in E(K)$ be $(S, \DD)$-integral points  
 with $\beta(P)= \beta(P') \in F_q(B,S)$ and suppose that  
there exists $u \in \Aut_B(X)=\Aut_K(E)=E(K)\rtimes \{\pm\Id\}$ (cf. Theorem \ref{t:translation} and Lemma \ref{l:auto-elliptic}) such that 
\[
h_P=u\circ h_{P'}
\]
where $f_P\colon Y \xrightarrow{h_P}  X\to B$ and $f_{P'}\colon Y' \xrightarrow{h_{P'}} X \to B$. 
Since the horizontal parts of the branch loci of $Y \to X$ and of $Y' \to X$ are respectively 
$(P) \cup \DD$ and $(P') \cup \DD$, 
 we must have $u_E\{P',D\}=\{P,D\}$. We consider two cases:

\begin{enumerate}
\item
Suppose that $d\geq 2$.  
We deduce that $u_E(P')=P$ and $u_E(D)=D$. 
 Hence, Lemma \ref{l:auto-elliptic} implies that there is at most one rational point $R\in E(K)$ modulo $E[d]$ 
which depends only on $\mathcal{L}(D)$ such that $u_E$ is of the form $\tau_{U}$ or $(-\Id)\circ \tau_{R+U}$ where $U \in E[d]$. 
As $u_E(P')=P$, there exists at most $2\#E[d]=2d^2$ possibilities for $P'$. 
 \item
Suppose that $d=1$. 
Lemma \ref{l:integral-point-extreme-case} implies that $D\in E(K)$ is a fixed rational point and $\DD$ is a section. 
Since $\Aut_K(E)=E(K)\rtimes \{\pm\Id\}$ and $u_E\{P',D\}=\{P,D\}$, 
we need to find $R \in E(K)$ verifying one of the following conditions: 
\begin{enumerate}
\item
$R+D= D, R+P'=P$;  
\item
$R+D=P, R+P'=D$; 
\item
$- (R+D) = D, -(R+P')=P$; 
\item
$-(R+D)=P, -(R+P')=D$.  
\end{enumerate}
Each case gives at most one choice for $R$ by the first equation and hence at most one choice for $P'$ by the second one. 
Thus, there are at most 4 possibilities for $P'$ in total.   
\end{enumerate}

We have just shown that the fibres of $\beta$ are uniformly bounded by $4d^2M(q,d+1)$.  
Since $F_q(B,S')$ is also uniformly bounded by a function $C(q,g, \#S')$ (cf. Theorem \ref{t:caporaso}),  
the number of $(S,\DD)$-integral points is bounded uniformly by 
$4d^2M(q,d+1)C(q,g,\#S')$, 
which a function depending only on $g$, $s$, $t$, $d$, and $\#\DD_{ram} \cup \DD_{sing}$ 
by the relations \eqref{e:s'-parshin-arakelov-integral} and \eqref{e:q-parshin-arakelov-integral}. 
The proof is thus completed. 
\end{proof}

\section{A negative result on the Parshin-Arakelov theorem}
\label{section:negative-parshin-arakelov-1}

Keep the notations as in Theorem \ref{t:uniform-caporaso-elliptic-1}, 
it is natural to attempt to generalize the above proof of Theorem \ref{t:uniform-caporaso-elliptic-1} 
to show, for example, that for 
every $n \in \N$, the following  
union of integral points of bounded denominators 
\begin{equation}
\label{e:I-n}
I_n = \cup_{S \subset B, \# S \leq n} \{(S, \DD)\text{-integral points}\} 
\end{equation}
is finite (which is true by Corollary \ref{c:height-iso-elliptic-silverman}). 
If we follow the same steps as in the proof of Theorem \ref{t:uniform-caporaso-elliptic-1}, 
we will then obtain a map with uniformly bounded finite fibres 
\begin{align*}
\cup_{S \subset B, \# S \leq n} \{(S,\DD)\text{-integral points}\} &\xrightarrow{\,\,\, \beta' \,\,\,}  \cup_{S' \subset B, \# S' \leq n'} F_q(B,S')  \\
P &\longmapsto (f_P \colon Y \to X \to B)
\end{align*}
for some  $n'  \in \N$.  Thus we reduce to prove the finiteness of 
$\cup_{S' \subset B, \# S' \leq n'} F_q(B,S')$. 
\par
Similarly, 
let $L$ be a very ample line bundle on $X$ and 
suppose that want to prove the finiteness of the following union of integral points 
with respect to a varying divisor
\begin{equation}
\label{e:J-L}
J_L =  \cup_{\DD \in |L|_{sm}} \{(S,\DD)\text{-integral points}\}, 
\end{equation}
where $|L|_{sm}$ denotes  
the open dense algebraic subset of $|L|$ consisting of 
smooth effective divisors (Bertini's theorem). 
Remark that $J_L$ is known to be finite  by Theorem \ref{t:height-iso-elliptic} using the base change 
to curves in $|L|_{sm}$ whose genus is constant.  
\par
Proceeding as in the proof of Theorem \ref{t:uniform-caporaso-elliptic-1}, we can also obtain a 
map 
\begin{align*}
\cup_{\DD \in |L|_{sm}} \{(S,\DD)\text{-integral points}\} 
&\xrightarrow{\, \beta'' \,}  
\cup_{S'' \subset B, \# S'' \leq n''} F_q(B,S'') \\
P &\longmapsto (f_P \colon Y \to X \to B)
\end{align*}
for some $n'' \in \N$ since the divisors $\DD \in |L|_{sm}$ are linearly equivalent, $\DD_{sing}=\varnothing$,  
and $ \# \DD_{rm}$ is uniformly bounded by the Riemann-Hurwitz formula (cf. \cite[Proposition 7.5.4]{liu-alg-geom}). 
It can be shown that the map $\beta''$ has uniformly bounded finite fibres 
(cf. the last part in the proofs of Theorem \ref{t:uniform-caporaso-elliptic-1} 
and Theorem \ref{t:negative-parshin-arakelov}). 
Therefore, we reduce again to show the finiteness of the set 
$\cup_{S'' \subset B, \# S'' \leq n''} F_q(B,S'')$. 
\par
Unfortunately, it turns out that the set $\cup_{S \subset B , \# S \leq s} F_q(B,S)$
is even very far from being finite because 
of the   strong negative finiteness result:

\begin{reptheorem}{t:negative-parshin-arakelov}
 For all large enough $q$ and $s$ depending only on the genus $g$ of $B$,  
\[
\cup_{S \subset B , \# S \leq s} F_q(B,S) 
\]
is uncountably infinite. 
Moreover, there exists $N(q,s,g) \in \N^*$, 
a Zariski dense open subset $I \subset \Proj^{N}$, and 
a map $I \to \cup_{S \subset B , \# S \leq s} F_q(B,S)$ with uniformly bounded finite fibres. 
 \end{reptheorem}

Therefore, additional arguments will certainly be needed with the above method to obtain known 
finiteness results for the above sets $I_n$, $J_L$ (cf. \eqref{e:I-n} and \eqref{e:J-L}). 
 \par
 
Before giving the proof of Theorem \ref{t:negative-parshin-arakelov}, 
we need to begin with several   lemmata. 

\begin{lemma}
\label{l:degree-to-P-1-at-most}
Let $B/k$ be a smooth projective integral curve of genus $g$ over an algebraically closed field $k$. 
Then there exists a finite $k$-morphism $f \colon B \to \Proj^1$ of degree at most $2g+1$. 
\end{lemma}

\begin{proof}
Let $b \in B$ and consider the effective divisor $D= (2g+1) [b]$. 
Since $\deg D=2g+1$, the line bundle $L=\OO(D)$ is very ample on $B$. 
By Riemann-Roch theorem, we have $\dim H^0(B, L)=g+2$.  
Fixing a basis of $H^0(B, L)$, we obtain an immersion embedding 
$j \colon B \to \Proj^{g+1}= \Proj |L|$. 
Let $z_0, \dots , z_{g+1}$ be the coordinate in $\Proj^{g+1}$. 
It is clear from the construction that for some $i$, the rational function 
$j^*z_i \colon B \to \Proj^1$ is non constant and of degree $\deg(j^*z_i)= 2g+1$.   
 \end{proof}
 
\begin{proof}[Proof of Theorem \ref{t:negative-parshin-arakelov}]
The idea of the proof is very similar to the proof of Theorem \ref{t:uniform-caporaso-elliptic-1}.  
Consider a non-isotrivial elliptic surface $f \colon X \to B$. Such a surface always exists 
since we can obtain one after a finite base change $B \to \Proj^1$ 
of the Legendre family $\mathcal{E} \to \Proj^1$ of elliptic curves defined in the affine plane by the equation 
$$
y^2= x(x-1)(x- \lambda),
$$ 
where $\lambda$ denotes the inhomogeneous coordinate on $\Proj^1$.   
By Theorem \ref{t:isotrivial-dominant}, the obtained surface $X$ is non-isotrivial 
since $\mathcal{E}$ is non-isotrivial. 
The model $\mathcal{E}$ can be taken to be for example the minimal elliptic surface 
associated to the generic elliptic curve $\mathcal{E}_{\C(\lambda)}$ over the function field $\C(\lambda)$. 
By Lemma \ref{l:degree-to-P-1-at-most}, the finite morphism $B \to \Proj^1$ can be taken 
to have degree $m \leq 2g+1$. 
As $\mathcal{E}$ has only three singular fibres lying above $\lambda=0, 1, \infty$, 
It follows that the set $T \subset B$ supporting singular fibres of $X$ has no more than $3(2g+1)$ points. 
\par
Let $H$ be a very ample line bundle on $\mathcal{E}$ and let $L$ be its pull back to $X$, which is also very ample. 
 Let  $d \in \N^*$. 
By Bertini's theorem, we obtain a Zariski open dense subset $I \subset \Proj |2dL|$ 
parametrizing an uncountable family of smooth and irreducible 
horizontal curves $(C_i)_{i \in I}\subset X$ belonging to the complete linear system $|2dL|$. 
 By the adjunction formula, the (arithmetic) genus $g_i$ of the curves $C_i$'s is a constant 
\begin{equation}
\label{generalized-riemann-hurwitz-apply-parshin}
g_i= \frac{C_i(C_i+K_X)}{2}+1= \frac{2dL(2dL+K_X)}{2}+1= m\frac{2dH(2dH+K_{\mathcal{E}})}{2}+1
\end{equation}
which is bounded only in terms of $g$ since $m \leq 2g+1$.  
 From \eqref{generalized-riemann-hurwitz-apply-parshin} and the  
Riemann-Hurwitz Formula  (cf. \cite[Proposition 7.5.4]{liu-alg-geom}), 
it is clear that for every $i \in I$, 
the degree of the ramification divisor $C_{i, ram}$ of the induced cover $C_i^{hor} \to B$ 
 is bounded in function of $g$. 
 Define $W_i= f(C_{i,ram}) \subset B$. 
It follows that $\#W_i$ is finite and bounded only in terms of $g$.  
\par
We can now obtain an uncountable number of pairwise non-equivalent   
non-isotrivial minimal families $Y_i\to B$ of curves of genus $q \geq 2$ 
which factor through $X\to B$ as follows. 
Let $X_b$ ($b \in B$) and $F$ be general fibres of $X$ and $\mathcal{E}$ respectively. 
Denote 
$$
n = L \cdot X_b = m (H \cdot F) \leq (2g+1)(H \cdot F).
$$   
By Proposition \ref{p:cyclic-covers}, we obtain  
a cyclic double cover $X'_i \to X$ of surfaces associated to the data 
$\OO(C_i) \sim 2ndL$. 
Let $Y_i \to X'_i$ be the strict resolution of singularity of $X'_i$. 
We blow down in $Y_i$ any $(-1)$-curves if they exist. 
Denote $f_i\colon Y_i \to X'_i \xrightarrow{} X \to B$ the induced composition. 
By the Riemann-Hurwitz formula applied to the general fibres $Y_{i, b} \to X_b$, 
we find that $f_i$ is a family of curves of genus 
$q$ satisfying $2q-2=2nd$ thus $q=nd+1 \geq 2$. 
\par
 Let $b \in B \setminus (W_i \cup T)$ where we recall that $W_i= f(C_{i,ram})$.   
Then $X_b$ is a smooth curve and $C_{i,b} \subset f^{-1}(b)$ is a smooth divisor by construction. 
By Proposition \ref{p:cyclic-covers-fonctorial}, 
$Y_{i,b}= f_i^{-1}(b) \to X_b$ is also the degree $d+1$ simple cyclic cover 
associated to the data $C_{i,b} \sim (2dL)_b$. 
Hence, Proposition \ref{p:cyclic-covers} implies that the 
fibre $Y_{i,b}= f_i^{-1}(b)$ is smooth whenever $b \in B \setminus (W_i \cup T)$. 
By Theorem \ref{t:isotrivial-dominant}, $Y_{i}\to B$ is nonisotrivial since $X$ is nonisotrivial by hypothesis. 
 Therefore, $Y_i$ is a minimal nonisotrivial family over $B$ of type $W_i \cup T$ of curves of genus $q$ and  
$$
Y_i \in F_q(B,W_i \cup T) \subset \cup_{S \subset B , \# S \leq s} F_q(B,S) 
$$ 
where the number $s$ depends only on $g$. 
To summarize, for the Zariski dense open locus of 
integral smooth curves $I \subset \Proj |2dL|$, 
we have defined a map 
$$
I  \to  \cup_{S \subset B , \# S \leq s} F_q(B,S), \quad i  \mapsto (Y_i \to B).
$$ 
To finish, we shall show that for any fixed $i \in I$, there are at most a uniformly bounded number (in terms of $g$) of $j \in I$ such that 
$Y_j\simeq Y_i$ over $B$. 
  \par
Theorem \ref{c:kani} tells us that 
 there is no more than $M(q,2)$ possible double covers $Y_i\to X$ up to composition with an element of $\Aut_B(X)=\Aut_K(E)$. 
Therefore, it is enough to prove that  each such class of double covers 
 has at most $8d^2n^2$ covers of the form $Y_j \xrightarrow{h_j}X$ 
where $f_i\colon Y_i \xrightarrow{h_i}  X\to B$ and $f_{j}\colon Y_j \xrightarrow{h_{j}} X \to B$ are constructed as above.
 For such $j$, there exists $u \in \Aut_B(X)=\Aut_K(E)=E(K)\rtimes \{\pm\Id\}$ (cf. Theorem \ref{t:translation} and Lemma \ref{l:auto-elliptic}) such that $h_i=u\circ h_j$. 
 Since the horizontal parts of the branch loci of $Y_i \xrightarrow{h_i} X$ and of $Y_j \xrightarrow{h_j} X$ are respectively 
$C_i$ and $C_j$, 
we must have $u(C_j)=C_i$. 
Since $C_i \sim C_j$, Lemma \ref{l:auto-elliptic} implies that there is at most one rational point $R\in E(K)$ modulo $E[2nd]$ with $u$ of the form $\tau_{U}$ or $(-\Id)\circ \tau_{R+U}$ where $U \in E[2nd]$. 
It follows that there are at most $2\#E[2nd]=2(2nd)^2=8d^2n^2$ choices for such $u$ and thus 
at most $8d^2n^2$ possibilities for such $j$ as claimed since $C_j=u^{-1}(C_i)$. 
 Therefore, the theorem is proved.   
\end{proof}

\section{Application to unit equations} 
\label{s:parshin-arakelov-unit-equation}
We shall apply the covering method to give a new 
proof for the uniform finiteness of unit equation over function fields as follows.  
\par
Let's begin with a construction of covers associated to integral sections in a 
ruled surface.  
Let $f \colon X \to B$ be the trivial ruled surface over $B$ 
and let $\pi \colon X \to B$ be the second projection. 
Let $\DD$ be the closure in $X$ of an effective reduced divisor $D$ of odd degree $d \geq 3$ on $\Proj^1_K$. 
Let $P\in X_K(K) \setminus D$ with $(P)\subset  X$  the corresponding section $B \to X$. 
Denote  
\[
Z=(P)+\DD.
\] 
Consider the divisor $Z_K=D+ [P]$ on $\Proj^1_K$. 
Let $b\in S$ be a fixed point. 
Since $\deg(Z_K)=\deg(D)+1=d+1$ and since the group $\text{Pic}^0_K(\Proj^1_K)$ of equivalent classes of degree $0$ line bundles of $\Proj^1_K$ is trivial, 
we deduce that $Z_K\sim_K (d+1)[P]$. 
Since $k(X)=K(t)=K(\Proj^1_K)$, we can extend the linear equivalence on the surface $X$ up to a vertical divisor. 
Hence, we deduce that there exists a vertical divisor $F$ on $X$ such that 
\[
Z\sim  (d+1)(P)+F.
\] 
We can write $F=\sum_in_i \pi^*[b_i]$ for $b_i \in B$ and $n_i\in \Z$ where $[b_i]$ denotes the effective divisor associated to $b_i$. 
We have $\sum_i n_i =2m-r$ for some $m, r \in \N$ and $0 \leq r \leq 1$. 
Thus $\text{deg}(\sum_i n_i [b_i] - (2m - r)[b])=0$. 
By properties of the Jacobian $J(B)$, we find that $\sum_i n_i [b_i] \sim \OO(2m[b]+r[b])$. 
It follows that $F=\sum_in_i \pi^*[b_i] \sim 2\OO(m\pi^*[b])+r[b]$. 
Thus for 
$L=\OO( \frac{d+1}{2}  (P) + m \pi^*[b])$, we have 
\[
Z + r \pi^*[b]  \sim L^{\otimes 2}.  
\] 
 
\par
 
By Proposition \ref{p:cyclic-covers}, we obtain a double cyclic cover $X' \to X$ 
associated to the data $(Z+ r\pi^*[b], L^{\otimes 2})$. The cover $X' \to X$ is  totally ramified above $Z+ r \pi^*[b]$. 
Moreover, it is smooth above $B\setminus S$ whenever $(P)$ is $(S, \DD)$-integral.  
By a strong resolution of singularity and by contracting all possible $(-1)$-curves,  
we obtain a minimal family of curves  
\[
f_P\colon Y_P\to X\to B.  
\]

\begin{definition}  
We say that $P$ is \emph{$\DD$-nontrivial} (or simply \emph{nontrival} 
if no there is no possible confusion)  
if the induced family $f_P \colon Y_P \to B$ is nonisotrivial.   
\end{definition}

\begin{example}
When $\DD$ is the sum of the three sections $(0)$, $(1)$, $(\infty)$, a 
rational point $x \in X(K)$ is $\DD$-nontrivial if and only if $x \notin \C^*$.    
\end{example}

The main result of the section is the following. 

\begin{theorem}
\label{t:unit-general-3}
Let the notations be as above.  
There is a function $U \colon \N^5 \to \N$ such that the following set of integral points 
\begin{align*}
 \{\text{nontrivial }(S, \DD)\text{-integral points of }X \to B \}
\end{align*}
has at most $U(g,s,d,d_{sing}, d_{ram})$ elements where $d_{sing}$, $d_{ram}$ 
denote the number of singular points on $\DD$ and the number of ramified points of the induced 
degree-$d$ cover $\DD \to B$. 
\end{theorem}

From Theorem \ref{t:unit-general-3}, we   recover the 
following known uniform  finiteness result on unit equation over function fields. 

\begin{corollary}
\label{c:unit-uniform-consequence}
The number of the solutions $(x,y)$ with $x/ y \notin \C$ of the $S$-unit equation 
$x+y=1$, $x, y \in \OO_S^*$ is uniformly bounded in terms of $g,s$. 
\end{corollary}

\begin{proof}
It suffices to show that the number of $x \in \OO_S^*$ with $1-x \in \OO_S^*$ and $x \notin \C$ is uniformly bounded. 
Such a point $x$ is exactly a nontrivial $(S,\DD)$-integral points in $X= \Proj^1 \times B$ 
with $\DD= (0)+(1)+(\infty)$. 
Since $\DD$ has no singular points and the cover $\DD \to B$ is \' etale, the corollary follows immediately from Theorem 
\ref{t:unit-general-3}.  
\end{proof}

\begin{remark}
\label{r:application-parshin-arakelov-to-unit-evertse} 
The first result on unit equations traces back to the work of Lang (cf.  \cite{lang-60}). 
In fact, Theorem \ref{t:unit-general-3} and 
Corollary \ref{c:unit-uniform-consequence}   
are consequences of the following remarkable result of 
Evertse (cf. \cite{evertse-86}) whose proof uses 
height theory combined with the   gap principle: 

\begin{theorem*}
[Evertse]
Let $B$ be a smooth projective curve over a field $k$ of characteristic $0$. 
Let $K=k(B)$ be the function field of $B$ and $S \subset B$ a finite subset. 
The set of  $(x, y) \in (\OO_{K, S}^*)^2$ with $x/y \notin k$ and 
$x+y=1$ has at most $2 \times 7^{2\#S}$ elements.  
\end{theorem*}

Evertse's effective bound is totally independent of the function field $K$. 
To obtain Theorem \ref{t:unit-general-3}, 
it suffices to make the base change $h \colon \tilde{D} \to B$ of degree 
$\deg(h)=d= \deg D_K$ where $\tilde{D}$ 
is the normalization of $\DD$. 
The set 
of nontrivial $(S, \DD)$-integral points of $X \to B$ 
then becomes a subset of the set of solutions of the $S'$-unit equation 
$x+y=1$ where $S'= h^{-1}(S)$ so that $\# S' \leq \deg(h) \# S= ds$. 
\end{remark}

However, our proof of Theorem \ref{t:unit-general-3} below does not use height theory and 
thus we give a new proof of a weak (but nontrivial) version of Evertse's theorem. 
\par
Remark  first that since an isomorphism of $\Proj^1$, i.e., a M\" obius transformation, 
is completely determined by the images of three distinct points.  
We can thus easily obtain the following elementary lemma. 

\begin{lemma}
\label{l:unit-non trivial}
Let $P,Q\in \Proj^1\backslash \{0,1,\infty\}$ be two distinct points.  
Then at most $4!=24$ isomorphisms of $\Proj^1$ sends the set 
$\{0,1,\infty,P\}$ to the set $\{0,1,\infty, Q\}$. 
\end{lemma}

\begin{proof}[Proof of Theorem \ref{t:unit-general-3}]  
Let $P\in X(K)$ be a nontrivial $(S,\DD)$-integral point of $X$. 
Then $Z=(P)+\DD$ has degree $d+1$ on general fibres of $Y_P \to B$. 
By the Riemann-Hurwitz formula applied to the double cover $Y_{P, t} \to X_t$ of general fibres ($t \in B$), 
$Y_P \to B$ is a family of curves of genus $q$ such that 
$2q-2=2(-2)+d+1$ so that $q=(d-1)/2$.  
By Proposition \ref{p:cyclic-covers} and Proposition \ref{p:cyclic-covers-fonctorial}, $Y_P \to B$ is of type 
$S'=S \cup \pi(\DD_{sing} \cup \DD_{ram})$ where 
$\DD_{sing}$ is the set of singular points of $\DD$ and $\DD_{ram} \subset \DD$ is the set of 
ramified points of the finite cover $\DD \to B$. 
\par
Since $P$ is a nontrivial $(S,\DD)$-integral point, the family $f_P$ is nonisotrivial and has good reductions  
outside of $S$. 
To summarize, we have constructed a  
map  
\begin{align*}
 \{\text{nontrivial }(S, \DD)\text{-integral points of }X \to B \} &\xrightarrow{\,\,\, \gamma\,\,\,} F_{q}(B,S') \\
P &\longmapsto (f_P \colon Y_P \xrightarrow{h_{P}}  X \to B)
\end{align*}
where $F_{q}(B,S')$ denotes the set of non equivalent classes of minimal families of curves of genus $q$ of type $S'$. 
Since $q \geq 1$ as $d \geq 3$, Shafarevich theorem (Theorem \ref{t:uniform-shafarevich-prob}) and the Parshin-Arakelov theorem (Theorem \ref{t:mordell-uniform-caporaso}) imply that $\# F_{q}(B,S')$ 
is uniformly bounded in terms of $q, g, \# S'$. 
To finish, we need to show that the map $\gamma$ above has uniformly bounded fibres. 
For this,  let $f_P \colon Y_P \xrightarrow{h_{P}}  X \to B$ be in the image of $\gamma$ for a some integral point $P$.   
We distinguish two cases. 
\par
Suppose first that $d \geq 5$ so that $q \geq 2$. 
 Theorem \ref{t:buj-gro} implies that up to an isomorphism of $\Proj^1$, 
there exists only a uniformly bounded finite number of double covers $Y_P \to X$. 
Thus, it suffices to show that if $Q$ a nontrivial $(S, \DD)$-integral point such that 
$h_P= \mu \circ h_Q$ for some $\mu \in \Aut_B(X)= \mathrm{PGL}_2(\C)$, then there are at most 24 choices for $Q$. But since $\mu$ must send the branch points of $h_Q$ to branch points of $h_P$, 
we have $\mu\{Q, D\}=\{P, D\}$ and   there are indeed 
no more than 24 possibilities for $Q$ by Lemma  \ref{l:unit-non trivial}. 
\par
Suppose now that $d=3$ so that $q=1$ and $Y_P \to B$ is an elliptic surface. 
It is well-known that up to an isomorphism of $\Proj^1$ and an automorphism of $(Y_P)_K$, 
there is only one double cover from $(Y_P)_K$ to $\Proj^1$ given by 
$(x,y) \mapsto y$ where $Y_P$ is given by $y^2=x^3+Ax+B$ for some coordinates $x,y$. 
Hence, as in the case $d \geq 5$, at most $24$ integral points $Q$ give rise to the same family of curves $Y_Q\simeq Y_P$. 
The conclusion thus follows. 
\end{proof}

\section{Tautological inequality and integral points}
\label{s:tautological}
We fix throughout an integral   quasi-projective variety $Z/k$ over 
an algebraically closed field $k$ of characteristic $0$. 
The language of the intersection theory is used freely (cf. \cite{fulton-98}). 
 \begin{definition}
\label{d:maximal-variation-family}
A family of $k$-elliptic surfaces $\XX \to \CC \to Z$ is called \emph{a relative maximal variation family} if 
\begin{enumerate} [\rm (a)]
\item
$ \CC \to Z$ is a family of smooth projective curves; 
\item
$\XX_z \to \CC_z$ is an elliptic surface for every $z \in Z$; 
\item
for every $z \in Z$ (not necessarily closed) and for 
$\xi $ the generic point of $\CC_z$, 
the Kodaira-Spencer class  
of the curve  
$X_\xi / \kappa(\xi)$ is nonzero.
\end{enumerate}

\begin{remark}
\label{r:parameter-tautological}
Let $z \in Z$. 
Let $g$ be the genus of  $\CC_z$ and 
let $\chi$ be the  Euler-Poincar\' e characteristic 
$\chi(\OO_{\XX_z})$ of $\XX_z$. 
 Since $Z$ is integral, $g, \chi$ are independent of $z \in Z$. 
\end{remark}

\end{definition}
 
Let $ \chi, g \in \N$. 
The existence of relative maximal variation families  
of (resp. semistable) elliptic surfaces 
with section which parametrize  
all elliptic surfaces $X \to C$ such that $\chi(\OO_X)=\chi$ 
and $g(C)=g$ is proved in \cite[Theorem 7]{seiler-1}.  

\begin{definition}
\label{d:adaptive-family-of-divisor}
Consider a family of elliptic surfaces $\XX \xrightarrow{f} \CC \to Z$ with a section. 
Let $T \subset \XX$ be the reduced divisor of singular fibres of $f$. 
An integral divisor $\DD \subset  \XX$ is an   
\emph{adaptive family of ample divisors}  
if $\DD + T$ is simple normal crossing and for every $z \in Z$: 
\begin{enumerate} [\rm (a)] 
\item 
$\DD_z \coloneqq \DD \cap \XX_z$ is an effective ample divisor on $\XX_z $;  
\item 
$\DD_z +T_z$ is a simple normal crossing divisor.  
\end{enumerate}
\end{definition}

By means of the tautological inequality 
(cf. Section \ref{tautological-inequaltiy-preliminary}), 
we shall give a proof of:   

\begin{reptheorem}{t:general-theorem-shorst-list-1}
Let $\XX \xrightarrow{f} \CC \to Z$ be a relative maximal variation family of    semistable 
elliptic surfaces with a zero section $O \colon \CC \to \XX$. 
Let $\DD \subset \XX$ be an adaptive family of ample effective divisors. 
 There exists   $c_1,c_2 >0$ such that 
for every $z \in Z$ and every $P \in \XX_z(k(\CC_z))\setminus \DD_z$, 
 \begin{align}
\label{e:tautological-uniform-bound-height-main}
\widehat{h}_{O_z} (P) \leq c_1s + c_2, \quad \text{ where } s= \# \sigma_P(\CC_z) \cap \DD_z.  
\end{align}
Here, $\widehat{h}_{O_z}$ is the N\' eron-Tate height on $\XX_z(k(\CC_z))$ 
associated to the origin $O_z$ and $\sigma_P \in \XX_z(\CC_z)$ is the corresponding section of $P$.  
\end{reptheorem} 
 
In fact, we can obtain with the tautological inequality that:  

\begin{theorem}
\label{t:c-12-independent} 
If $\DD= (O)$ is the zero section, the conclusion of Theorem \ref{t:general-theorem-shorst-list-1} still holds. 
Moreover, the constants $c_1, c_2$ 
depend only $g,\chi$ (cf. Remark \ref{r:parameter-tautological}). 
 \end{theorem} 
  
\subsection{Uniform bound of the Mordell-Weil rank in families}

The following lemma is well-known to experts. 

\begin{lemma}   
\label{l:rank-bound}
Let $f \colon X \to B$ be a nonisotrivial elliptic surface over a smooth projective curve $B$.  
 Then  $ r= \rank X(B) \leq \rho(X) \leq 12\chi(\OO_X)+4g(B)-2$, 
  where   $g(B)$ is the genus of $B$, $\rho(X)$ is the Picard number , 
and $\chi (\OO_X)$ is the Euler-Poincar\' e characteristic of $X$.
\end{lemma}

 \begin{proof}
The Shioda-Tate formula  
(cf. \cite[Theorem 1.3, Corollary 5.3]{shioda-tate-formula} or \cite{shioda-schutt-lecture}) 
implies that  $r= \rank X(B) \leq  \rho(X)$. 
Let $e(X)$ denote the topological Euler characteristic of $X$, then 
\begin{equation}
\label{e:rank-bound-proof-euler-betti}
 b_2(X)=e(X)-2+2b_1(X)=e(X)-2+4g(B).
 \end{equation} 
The first equality follows from the Poincar\' e duality. 
The second equality follows from: 
 $$b_1(X)=b_1(B)=2g(B).$$ 
In fact, as $f_*\OO_X=\OO_B$, $R^1f_*\OO_X=\mathbb{L}^{-1}$ with 
$\mathbb{L}$ the fundamental line bundle of $X$ satisfying  
$\deg \mathbb{L} =  \chi (\OO_X) >0$ (since $X$ is nonisotrivial, cf. \cite{Miranda}), the Leray Spectral sequence: 
$0 \to H^1(B, \OO_B) \to H^1(X, \OO_X)  \to H^0(B, \mathbb{L}^{-1})= 0$ implies that $b_1(X)=b_1(B)$.  
\par 
The Dolbeault isomorphism says that 
$H^1(X, \OO_X) \simeq H^{0}(X, \Omega_X^1)$ and $ H^1(B, \OO_B) \simeq H^{0}(X, \Omega_B^1)$. 
Hence, by Hodge decomposition, we have $b_1(X)=b_1(B)$ since:  
$$b_1(X)=h^{1,0}+h^{0,1}=2h^{0,1}=2h^1(X,\OO_X), \quad 
b_1(B)=h^{1,0}+h^{0,1}=2h^{0,1}=2h^1(B, \OO_B)$$

  From the injectivity of the cycle class map $\NS(X) \to H^2(X, \Q)$,  
 $\rho(X) \leq b_2(X)$. 
 On the other hand,  Noether's formula $\chi(\OO_X)=(K_X^2+e(X))/12$ 
 implies that $12\chi(\OO_X)=e(X)$ since $K_X^2=0$.  Therefore,  
 \eqref{e:rank-bound-proof-euler-betti} implies that: 
 $ \rho(X) \leq b_2(X)=12 \chi(\OO_X) -2 +4g(B)$.    
 \end{proof}

\subsection{The tautological inequality}
\label{tautological-inequaltiy-preliminary}

Let $X$ be a smooth projective variety over a field $k$. 
Let $D \subset X$ be a \emph{simple normal crossing} divisor, i.e., 
in some local coordinates $z_1, \cdots, z_n$ at each point $x \in X$, 
$D$ is given by an equation of the form $z_1\cdots z_k=0$ with $k \leq n$. 

\begin{definition} 
\label{d-log-differential-intro}
The sheaf of differentials $V_1= \Omega_{X/k}(\log D)$ with logarithmic poles along $D$ 
is well-defined vector bundle. 
It is given locally at $x \in X$ by $\frac{dz_1}{z_1}, \cdots, \frac{d{z_k}}{z_k}, dz_{k+1}, \cdots, d{z_n}$. 
$X_1(D) \coloneqq \Proj V_1$ is defined as $\mathrm{Proj} \bigoplus_{d \geq 0} S^dV_1$. 
\end{definition}
 
We have a canonical morphism $\pi_1 \colon X_1(D) \to X$. 
Denote by $\OO(1)$ the tautological line bundle on $X_1(D)$. 
Consider a $k$-morphism $f \colon Y \to X$ where $Y/k$ is a smooth projective curve. 
Assume that $f(Y) \not \subset D$. 
Then $f$ can be lifted to a morphism  
$f' \colon Y \to X_1(D)$ such that $\pi_1 \circ f' = f$ as follows. 
Sine the pullback of every logarithmic differential with poles along $D$ is  a logarithmic differential with  
poles along the support $(f^*D)_{red}$, we have a natural map 
$f^*(V_1) \xrightarrow{f^*} \Omega_{Y/k}((f^*D)_{red})$.  
This gives rise to a quotient map $f^*(V_1) \to f'^* \OO(1)$ which defines  
the map $f'$ by the universal property of $\Proj V_1$.  
We obtain an inclusion of sheaves 
$$ 
f'^* \OO(1) \hookrightarrow \Omega_{Y/k}((f^*D)_{red}).
$$ 
For $g(Y)$ the genus of $Y$, it follows that 
\begin{equation}
\label{e:tautological-inequality}
  \deg f'^* \OO(1)  \leq \deg \Omega_{Y/k}((f^*D)_{red}) = 2g(Y) - 2 + \deg (f^*D)_{red}, 
\end{equation}

\begin{definition}
\label{d:tautological-inequality}
We call \eqref{e:tautological-inequality} the \emph{tautological inequality} associated to 
$f \colon Y \to X$.  
\end{definition}


\subsection{Ampleness on ruled surfaces}

We continue with a useful criterion of ampleness. 

\begin{lemma}
\label{l:ample-isotrivialitity-elliptic}
 Let $V$ be a vector bundle of rank $2$ on a smooth projective curve $X$ 
over a field $L$ of characteristic $0$. 
Assume that $V$ satisfies a non splitting short exact sequence: 
\begin{align} 
\label{lemma-exact-sequence-lemma-tautological}
0 \to \OO_X \to V \to \OO(D) \to 0
\end{align}
where $D$ is an effective divisor on $X$ such that $\deg D >0$. 
Then the tautological bundle $\OO(1)$ on the ruled surface $\Proj_X(V)$ 
is ample. 
 \end{lemma}

\begin{proof} 
The class of the exact sequence \eqref{lemma-exact-sequence-lemma-tautological} 
in $H^1(X, \OO(D))$ is nonzero if and only if its 
class in $H^1(X_{\bar{L}}, \OO(D_{\bar{L}}))$ is non zero. 
The Cohomological criterion for ampleness as in \cite[Proposition III.2.6.1]{ega-4-3} 
and the invariant of cohomology under flat base change allows us to suppose that 
$L$ is algebraically closed. 
We can now apply the Nakai-Moishezon criterion to prove that $\OO(1)$ is ample.  
The proof goes as in \cite[Lemma 6.27]{gasbarri}. 
\end{proof}

\subsection{Canonical height and intersection theory} 
To transfer bounds on the intersection inequality 
obtained by the tautological inequality to bounds on the canonical 
height, we use the following comparison 
result:  

\begin{lemma}
\label{l:bound-elliptic-canonical-height}
Let $X \to B$ be an elliptic surface over a curve $B$ with a zero section $(O)$. 
Let  $K=k(B)$. 
 Then every rational point $P \in X(K)$ satisfies 
\begin{equation}
\label{e:compare-canonical-height} 
|\widehat{h}_X(P) -   (P) \cdot (O) | \leq - (O)^2 = \chi(\OO_X). 
\end{equation}

\end{lemma}

\begin{proof}
See for example \cite[Lemma 3]{elkies}. Here, we normalize 
the canonical height $\widehat{h}$, 
which may differ by a factor of $2$ to some definitions in the literature, 
so that \eqref{e:compare-canonical-height} holds. 
 \end{proof}

Return now to the notations of Theorem \ref{t:general-theorem-shorst-list-1}. 

  \subsection{Preliminary reductions} 
Since $k$ is of characteristic zero, $Z$ 
admits an integral resolution of singularities $Z^{sm}$ (cf. \cite{hironaka}). 
Up to making a base change of $\XX \to \CC \to Z$ and of $O$ 
to $Z^{sm}$, we can assume without loss of generality that 
$Z$ is smooth. It follows that 
$\CC$ and thus $\XX$ are smooth varieties. 
Denote $(O) \subset \XX$ the image of the section $O$ 
equipped with the reduced scheme structure. 
Then $(O)$ 
is a smooth divisor of $\XX$. 
We can clearly assume moreover that $\XX$ and $\CC$ are integral. 
 \par
Let $F \subset \CC$ be the effective reduced divisor of singular locus of the morphism 
$f \colon \XX \to \CC$. 
Then $T=f^*F$ is the divisor of singular fibres of $f \colon \XX \to \CC$. 
\par
Consider the embedded resolution of singularities $\mu \colon \XX' \to \XX$  
of the effective Cartier divisor $(O)+T$ in $\XX$ (cf. \cite{hironaka}, see also 
\cite[Theorem 4.1.3]{lazarsfeld-positivity-I}). 
The map $\mu$ can be obtained as a finite 
sequence of blowups along smooth centers
supported in the singular loci of $(O) + T$ (thus contained in $T$). 
 Up to replacing $\XX$ by $\XX'$ and $(O)+T$ by the support of $\mu^*((O)+T)+E$ where $E$ is the exceptional divisor 
of $\mu$, 
we can   suppose that the divisor 
 $$
D \coloneqq (O)+ T
$$ 
is   simple normal crossing. 
As the family $\XX \to \CC$ is semistable and admits a section 
by hypothesis, 
the fibres $D_z$ are also simple normal crossing 
for all $z \in Z$. 
\par
Therefore, the logarithmic cotangent bundles  
$\Omega_{\XX} (\log D)$ and 
$\Omega_{\XX_z} (\log D_z)$ are well-defined for every $z \in Z$ 
(cf. Definition \ref{d-log-differential-intro}). 
For the notations, we denote for every $z \in Z$, 
\begin{equation}
\XX(D) \coloneqq \Proj_{\XX} (\Omega_{\XX/Z} (\log D) ) \xrightarrow{\pi} \XX 
, \quad \quad
\XX_z(D_z) \coloneqq  \Proj(\Omega_{\XX_z} (\log D_z)) \xrightarrow{\pi_z} \XX_z.
\end{equation}

Here, $\Omega_{\XX/Z} (\log D)$ denotes the relative logarithmic cotangent bundle which 
by definition fits in the following short exact sequence: 

\begin{equation} 
0 \longrightarrow f^*\Omega_Z  
 \longrightarrow  \Omega_{\XX}(\log D ) 
 \longrightarrow  \Omega_{\XX/Z}(\log D)  \longrightarrow 0.
 \end{equation}

\subsection{Main induction step} 
\label{tautological-main-induction}

 Let $\eta \in Z$ be the generic point of $Z$. 
 Consider the universal elliptic surface $\XX_\eta \to \CC_\eta$   over the universal curve $\CC_\eta$. 
 Let $ \mathbb{L}=(R^1f_{\eta *}O_{\XX_\eta})^{-1}$ be the 
fundamental line bundle of $\XX_\eta \to \CC_\eta$. Then $\deg( \mathbb{L}) = \chi(\OO_X) > 0$ since 
$\XX_\eta \to \CC_\eta$ is nonisotrivial 
 and $\omega_{\XX_\eta}=f_\eta^*( \mathbb{L}  \otimes \omega_{\CC_\eta})$ 
(cf. \cite{Miranda}).  

\begin{lemma} 

We have an exact sequence of  {vector bundles}: 
\begin{equation} 
\label{l-extension-of-nef-tautologic}
0 \longrightarrow f_{\eta}^*\Omega_{\CC_\eta / \kappa(\eta)}(F_\eta) 
 \longrightarrow  \Omega_{\XX_\eta /\kappa(\eta)}(\log D_\eta)  
 \longrightarrow  \OO_{\XX_\eta}((O)_\eta) \otimes f_\eta^*(  \mathbb{L}) 
 \longrightarrow 0.
 \end{equation}  
\end{lemma}

\begin{proof} 
Let $W$ be the quotient of $\Omega_{\XX_\eta /\kappa(\eta)}(\log D_\eta) $ 
by $f_{\eta}^*\Omega_{\CC_\eta / \kappa(\eta)}(F_\eta)$. 
The local freeness of $W$ follows by a local calculation at points $P$  
lying on the divisor $T$ of singular fibres. 
Moreover, the exact sequence \eqref{l-extension-of-nef-tautologic} implies that:
\begin{align*} 
W &= \det(\Omega_{\XX_\eta /\kappa(\eta)}(\log D_\eta))\otimes \left(f_{\eta}^*\Omega_{\CC_\eta/ \kappa(\eta)}(F_\eta)\right)^{-1}\\
  &=f_\eta^*( \mathbb{L}\otimes \omega_{\CC_\eta}) \otimes \OO(D_\eta)\otimes f_\eta^*(\omega_{\CC_\eta}^{\otimes(-1)}) \otimes f_\eta^* \mathcal{O}(-F_\eta)
  \\
  &= \OO_{\XX_\eta}((O)_\eta) \otimes f_\eta^*(  \mathbb{L}). 
\end{align*} 
\end{proof} 
Let $\xi \in \CC_\eta$ be the generic point of $\CC_\eta$ and $\kappa(\xi)=\kappa(\eta)(\CC_\eta)$ the function field 
of $\CC_\eta$. 
Since $\XX \xrightarrow{f} \CC \to Z$ is of relative maximal variation, 
the Kodaira-Spencer class of $\XX_\eta/\kappa(\eta)$ 
is nonzero. In other words, the following exact sequence of vector bundles is non splitting: 
\begin{equation} 
0 \longrightarrow f_{\xi}^*\Omega_{\kappa(\xi) / \kappa(\eta)} 
 \longrightarrow  (\Omega_{\XX_\eta /\kappa(\eta)}(\log D_\eta))_\xi 
 \longrightarrow  \Omega_{\XX_\xi /\kappa(\xi) }(\log D_\xi) \longrightarrow 0.
 \end{equation}  
As $\kappa(\xi) / \kappa(\eta)$ is a separable 1-dimensional transcendental field extension, 
$f_\xi^*\Omega_{\kappa(\xi) / \kappa(\eta)}$ $\simeq \OO_{\XX_\xi}$. 
Since $\XX_\xi/\kappa(\xi)$ is an elliptic curve, 
$\Omega_{\XX_\xi /\kappa(\xi) }(\log D_\xi)= \OO_{\XX_\xi}(D_\xi)$. 
Hence, it follows from Lemma \ref{l:ample-isotrivialitity-elliptic} 
that the tautological line bundle $\OO_\xi(1)$ is ample on the elliptic ruled surface $\Proj (\EE_\xi) \to \XX_\xi$ 
where we define 
$$
\EE \coloneqq \Omega_{\XX_\eta/ \kappa(\eta)} (\log D_\eta). 
$$
\begin{lemma} 
\label{l:globally-generated-lemma-semistable} 
There exists an integer $N \geq 1$ and an effective divisor $\VV_\eta$ of $\CC_\eta$ 
such that the following line bundle on $\Proj (\mathcal{E}) = \XX_{\eta}(D_\eta)$ is globally generated:  
$$\LL_\eta \coloneqq \OO_{\Proj (\mathcal{E})}(N) \otimes \pi_\eta^*\OO(-D_\eta + f_\eta^*\VV_\eta).$$
\end{lemma}

\begin{proof}
As $\OO_\xi(1)$ is ample, the line bundle $\OO_\xi(N_1) $ 
is very ample on 
$ 
\Proj (\mathcal{E}_\xi)$ for some integer $N_1 \geq 1$. 
Take any {basis} $s_1, \dots, s_k$ of the linear system $|\OO_\xi(N_1)|$ on 
$\Proj(\mathcal{E}_\xi)$. 
As $\CC_\eta/\kappa(\eta)$ is a curve, 
there exists an effective divisor 
$\VV_1$ on $\CC_\eta$  
such that $s_1, \dots, s_k$ extend to global sections of 
$H^0(\Proj (\mathcal{E}), \OO_{\Proj (\mathcal{E})}(N_1)\otimes \pi_\eta^*\OO(f_\eta^*\VV_1))$. 
\par
 Since $\OO_\xi(N_1)$ is very ample on the generic fibre $\Proj (\mathcal{E}_\xi)$ 
of $\Proj(\mathcal{E})$, the line bundle 
 $ \OO_{\Proj (\mathcal{E})}(N_1) \otimes \pi_\eta^*\OO(\VV_1)$ is a big line bundle on 
 $\Proj (\mathcal{E})$  by the very definition 
 of bigness (its the augmented base locus cannot not be  
 all of $\XX_\eta(D_\eta)$, cf. \cite{birkar}).  
 \par 
 Since 
 $\deg( \mathbb{L}) = \chi(\OO_{\XX_\eta}) = - (O)_\eta^2 > 0$, 
 the line bundle $\OO_{\XX_\eta}((O)_\eta) \otimes f_\eta^*(\mathbb{L})$ 
 is nef on $\XX_\eta$. 
 On the other hand, 
$\deg( \Omega_{\CC_\eta / \kappa(\eta)}(F_\eta) )\geq 0$.  
Indeed, it is enough to consider the case $g=0$. But in 
this case, the Shioda-Tate formula \cite{shioda-schutt-lecture} 
implies that 
$2a+m \geq 2\chi(\XX_\eta)+2$ where $a$, $m$ 
are respectively the number of additive and multiplicative 
singular fibres of $\XX_\eta \to \CC_\eta$. 
Since   $\XX_\eta$ is nonisotrivial,  $\chi(\OO_{\XX_\eta})>0$ and 
it follows that $\deg F = a +m \geq 2$. 
Hence $f_\eta^*\Omega_{\CC_\eta / \kappa(\eta)}(F_\eta)$ is nef 
on $\XX_\eta$. 
\par
Therefore, $\Omega_{\XX_\eta /\kappa(\eta)}(\log D_\eta)$ 
is an extension of nef line bundles (cf. \eqref{l-extension-of-nef-tautologic}) 
thus it is a nef vector bundle on $\XX_\eta$.  
Thus, the line bundle $\OO_{\Proj (\mathcal{E})}(N_1)\otimes \pi_\eta^*\OO(\VV_1)$   is actually nef and big. 
\par
Since the pullback of $\OO_{\Proj (\mathcal{E})}(N_1)\otimes \pi_\eta^*\OO(\VV_1)$ to 
the generic fibre over $\CC_\eta$ is $\OO_\xi(N_1)$ which is  
very ample, Nakamaye's theorem \cite{nakamaye} (which is valid over any field, cf.  \cite{birkar}) 
implies immediately that  
the augmented base locus 
$B_+(\OO_{\XX_\eta(D_\eta)}(N_1)\otimes \pi_\eta^*\OO(\VV_1)) \subset \Proj (\mathcal{E})$ 
is vertical over $\CC_\eta$, i.e., it does not dominate $\CC_\eta$. 
By definition of the augmented base locus (cf. \cite{birkar}), 
there exists $N \geq 1$ and an effective vertical divisor $\VV_\eta$ such that 
$\mathcal{L}_\eta$ is globally generated as desired. 
\end{proof}

Let $\VV$ be any effective divisor on $\XX$ extending $\VV_\eta$. 
Since $\VV_\eta + T_\eta$ is vertical with respect to the projection $f_\eta \colon \XX_\eta \to \CC_\eta$, 
The pushforward $(f_\eta)_*(\VV_\eta + T_\eta)$ is an effective divisor on $\CC_\eta$ and we can 
set  
$M = \deg (f_\eta)_*(\VV_\eta +T_\eta) \in \N$.    
It follows that for every $z$ in some Zariski dense open subset $U_0 \subset Z$, 
we have 
\begin{equation}
\label{e:definition-of-M}
\deg (f_z)_*(\VV_z + T_z) = M.
\end{equation}
\par
Since $D \subset \XX$ and $D_z \subset \XX_z$ are  simple normal crossing, 
the formation of the relative logarithmic differential sheaves $\Omega_{\XX/ Z} (\log D)$ 
commutes with the localizations to fibres above points $z \in Z$. 
This means that for every point $z \in Z$,  
$(\Omega_{\XX/Z} (\log D))_z = \Omega_{\XX_z} (\log D_z)$. 
Moreover, 
we have a commutative cartesian diagram 
 \[
\label{d-abelian-diagram-etale}
\begin{tikzcd}
\XX_z(D_z) \arrow[d, hook]   \arrow[r, "\pi_z"] & \XX_z \arrow[r, "f_z"]  \arrow[d,hook]  & \CC_z \arrow[r, "h_z"]  \arrow[d,hook] &  z  \arrow[d,hook] \\  
\XX(D)  \arrow[r, "\pi"] & \XX \arrow[r, "f"]  & \CC \arrow[r, "h"] &  Z. 
 \end{tikzcd}
\]
\par
Let $\OO(1)$ be the tautological line bundle on $\XX(D)$ and consider 
the following line bundle:  
$$
\LL \coloneqq  \OO (N) \otimes \pi^*\OO(-D + \VV).
$$
By Lemma \ref{l:globally-generated-lemma-semistable}, 
$\LL_\eta$ is globally generated where $\eta$ is the generic point of $Z$. 
Denote $\varphi \coloneqq h \circ f \circ \pi \colon \XX(D) \to Z$ 
and $\varphi_z = h_z \circ f_z \circ \pi_z \colon \XX_z(D_z) \to z$ the induced structure morphism of 
$\XX_z(D_z)$ for every $z \in Z$. 
By the local constructibility of the set of $z\in Z$ satisfying the surjectivity of the  map 
\begin{equation}
\label{e:globally-generated-fiber-wise} 
 (\varphi^*\varphi_*\LL)_z = (\varphi_z)^* (\varphi_z)_* \LL_z \to \LL_z, 
 \end{equation}
 there exists a nonempty Zariski open subset   
$U \subset U_0 \subset Z$ such that 
the map \eqref{e:globally-generated-fiber-wise} is surjective, 
i.e., $\LL_z$ is globally generated, for every $z\in U$.  
\par
Therefore,  for every section $\sigma_P \in  \XX_z(\CC_z)$  
corresponding to a rational point $P \in \XX_z(k(\CC_z))$ for some $z\in U$,  
we have $\deg_B (\sigma_P')^*\LL_z \geq 0$,  
where $\sigma_P' \colon B \to \XX_z(D_z)$ is the derivative map which lifts $\sigma_P$ 
(cf. Defintion \ref{d:tautological-inequality}). 
\par
Since $\sigma_P = \pi_z \circ \sigma'_P$ and 
$\LL_z = \OO_{\XX_z(D_z)}(N) \otimes  \pi_z^*\OO(-D_z + \VV_z)$,  
 it follows immediately for every $z \in U$ that: 
 $$
D_z \cdot \sigma_P(B) \leq \deg (\sigma'_P)^* \OO_{\XX_z(D_z)}(N)  + \VV_z \cdot \sigma_P(B) . 
$$ 
Remark that $\sigma_P(B) \cdot F \leq 1$ for every integral fibre $F$ of 
$\XX_z \to \CC_z$. 
Combining with the tautological inequality applied to $P$ 
 (cf. Definition \ref{d:tautological-inequality}), 
we deduce that 
\begin{align}
\label{e:bounded-uniform-elliptic-height-1}
D_z \cdot \sigma_P(B) 
\leq N(2g - 2+ \# D_z \cap \sigma_P(B) ) + \# \VV_z \cap \sigma_P(B) .
\end{align}
Hence, by the definition of $M$ (cf. \eqref{e:definition-of-M}) 
and of $U$, we find for every $z \in U$ that: 
\begin{align*}
(O)_z \cdot \sigma_P(B) 
& \leq N(2g - 2 + \# (O)_z \cap \sigma_P(B) ) + \# \VV_z \cap \sigma_P(B)  +(N-1) \# T_z \cap \sigma_P(B) . 
\\
& \leq N(2g - 2 + \# (O)_z \cap \sigma_P(B) ) + NM. 
\end{align*}
On the other hand, we have 
 $|(O)_z\cdot \sigma_P(B)  - \widehat{h}_{O_z}| \leq - (O)_z^2 = \chi$ 
by Lemma \ref{l:bound-elliptic-canonical-height}. 
 Thus, we have proven the following main induction step:   
\begin{lemma}
\label{l:tautological-main-induction}
There exists a Zariski dense open subset 
$U \subset Z$ and   $M, N \in \N$ such that 
for every rational point $P \in \XX_z(k(\CC_z))$ with $z \in U$, we have: 
\begin{align}
\label{e:final-step-tautological}
\widehat{h}_{O_z}(P)  \leq N(2g - 2 + \# (O)_z \cap \sigma_P(B) ) + NM + \chi. 
\end{align}
\end{lemma}

\subsection{Proof of Theorem \ref{t:c-12-independent}}
Apply the above procedure \eqref{e:final-step-tautological} 
for each integral component of  $Z \setminus U$ and so on. 
Since $\dim Z \setminus U <\dim Z$, the   process has only finitely steps. 
It is then clear that there exists $c_1, c_2 >0$ such that for every $z \in Z$
and every section $\sigma_P \in \XX_z(\CC_z)$ associated to a rational 
point $P \in \XX_z(k(\CC_z))$,  
we have: 
 \begin{align}
\label{e:tautological-uniform-bound-height-main}
\widehat{h}_{O_z} (P) \leq c_1s + c_2, \quad \text{ where } s= \# \sigma_P(B) \cap (O)_z.  
\end{align}

By  (\cite[Theorem 7]{seiler-1}), 
the family $\XX \to \CC \to Z$ 
can be taken as a parameter 
space of all semistable elliptic surfaces with section of parameters 
$\chi, g$ (cf. Remark \ref{r:parameter-tautological}). 
The constants $c_1, c_2$ thus depends only on $\chi, g$ 
and the proof of Theorem \ref{t:general-theorem-shorst-list-1} is completed.

\subsection{Proof of Theorem \ref{t:general-theorem-shorst-list-1}}

By hypothesis, $\DD \subset \XX$ is an adaptive family of ample effective divisor. 
Thus, we can apply, \emph{mutatis mutandis}, 
the above \emph{Main induction step} \ref{tautological-main-induction}, by replacing $D$ by $\DD+T$, 
to obtain 
a Zariski dense open subset $U \subset Z$  
and   $M, N \in \N$ such that 
for every rational point $P \in \XX_z(k(\CC_z))$ with $z \in U$, we have 
(cf. \eqref{e:bounded-uniform-elliptic-height-1}): 
\begin{equation} 
\label{final-step-tautological-inequality}
\DD_z \cdot \sigma_P(B) 
\leq N(2g - 2+ \# \DD_z \cap \sigma_P(B) ) + M .
\end{equation}

To estimate the canonical heights, we need the following auxiliary result.  
 
\begin{lemma}
\label{l:uniform-height}  
There exists $A \in \N$ and a 
Zariski dense  
open subset $W \subset Z$ 
such that for every $z\in W$, the linear system $|A \DD_z - (O)_z|$ is base-point-free where $(O)_z$ is the zero section of
the elliptic surface $f_z \colon \XX_z \to \CC_z$. 
\end{lemma}

\begin{proof}
Consider   the divisors $\DD_\eta$ and $(O)_\eta$ on $\XX_\eta$ 
where $\eta$ is the generic point 
of $Z$.  
Since $\DD$ is a family of ample divisors by hypothesis, 
$\DD_\eta$ is   ample on $\XX_\eta$. 
Define $\LL(A) \coloneqq \OO(A \DD - (O))$ then 
there exists $A \in \N$ such that the line bundle 
$\LL(A)_\eta = \OO(A \DD_\eta - (O)_\eta)$ 
is globally generated on $\XX_\eta$. 
In other words, the canonical morphism
$ (f^*f_*\LL)_\eta = (f_\eta)^* (f_\eta)_*\LL(A)_\eta \to \LL(A)_\eta$ 
is surjective. 
Let $\phi \coloneqq h \circ f \colon \XX \to Z$ (cf. \eqref{d-abelian-diagram-etale}). 
It follows that the canonical morphism
$$ (\phi^*\phi_*\LL)_\eta = (\phi_\eta)^* (\phi_\eta)_*\LL(A)_\eta \to \LL(A)_\eta$$
is surjective. 
By the local constructibility property over 
the base of the surjectivity of morphism of coherent sheaves,  
there exists a Zariski dense open subset $W \subset Z$ such that 
for every $z\in W$, the following natural map is surjective: 
$$ (\phi_z)^* (\phi_z)_*\LL(A)_z \to \LL(A)_z.$$
It follows that the linear system $|A \DD_z - (O)_z|$ is  base-point-free for every $z\in W$.   
\end{proof} 
Denote $U'= W \cap U$. Then $U'$ is a Zariski dense open subset of $Z$. 
For every $z\in U'$,  
\begin{align*}
\label{e:final-step-tautological-last}
\widehat{h}_{O_z}(P) 
& \leq  (O)_z \cdot \sigma_P(B) + \chi &  (\text{by Lemma } \ref{l:bound-elliptic-canonical-height})
\\
 & \leq A  \DD_z \cdot \sigma_P(B) + \chi  &  (\text{by Lemma }  \ref{l:uniform-height})
\\
&  \leq AN(2g - 2 + \# \DD_z \cap \sigma_P(B)) + AM+ \chi. & (\text{by }   \eqref{final-step-tautological-inequality}) 
\end{align*}

We continue the above procedure 
for each integral component of $Z \setminus U'$ and so on. 
As $\dim Z \setminus U' <\dim Z$, the process terminates after a finite number of steps. 
Hence, there exists $c_1, c_2 >0$ such that for every $z \in Z$ and every  $P \in \XX_z(k(\CC_z))$, 
\begin{align*}
\widehat{h}_{O_z} (P) \leq c_1s + c_2, \quad \text{ where } s= \# \sigma_P(B) \cap \DD_z.  
\end{align*}
The proof of Theorem \ref{t:general-theorem-shorst-list-1} is thus completed.

\section{Appendix}

For ease of reading, we  reformulate several known results 
 necessary for our article.  
\subsection{de Franchis theorems}

We collect in this section several finiteness results concerning 
ramified coverings of curves. 
Let $k$ be a field 
and let $C/k$ be a smooth, geometrically connected, projective curve of genus $q$. 
Let $F=k(C)$ and   $m\in \N$. 
 Each subfield $k\subsetneq F'\subset F$ corresponds to an $k$-isomorphism class of a smooth, 
 geometrically connected, projective curve $C'$ such that the corresponding non constant map $f \colon C\to C'$ verifies $f^*(k(C'))=F'$.
\par
An immediate consequence of a result of Tamme-Kani  
is the following effective version of the de Franchis Theorem: 

\begin{theorem}
[Tamme-Kani] 
\label{c:kani}
Let $k$ be a field and let $m\in \N$.   
Let $E/k$ be an elliptic curve and let $C/k$ be a smooth  projective geometrically connected curve of genus $q$. 
Then, up to composition with an element of $\Aut_k(E)$, 
the number of degree-$m$ covers $h \colon C\to E$ 
is uniformly bounded by an effective function $M(q,m)$ depending only on $q$ and $m$. 
\end{theorem}

\begin{proof}
It suffices to apply  \cite[Theorem 4]{kani86} (and the corollary that follows) and define 
$M(q,m)= 2^{6q^2-1}m^{4q^2-2}(\zeta(2)/2)m^2+(m/2)(\log(m)+1).$   
\end{proof}

Given two covers $f_1 \colon X\to Y_1$ and $f_2 \colon X\to Y_2$ of compact Riemann surfaces, we say that $f_1$ and $f_2$ are \emph{equivalent} if there is a biholomorphic map $p\colon Y_1 \to Y_2$ such that $f_2=p\circ f_1$. 

\begin{theorem} [Bujalance-Gromadzki] 
\label{t:buj-gro} 
Given a compact Riemann surface $X$ of genus $q\geq 2$,  
the set of nonequivalent   double covers $X\to Y$ is bounded by a function $D(q)$. 
\end{theorem}

\begin{proof}
See \cite{buj-gro-00}. 
\end{proof}

\subsection{Shafarevich problems} 
Let $N \geq 1$. 
The moduli space   
of elliptic curves with level $N$-structure $Y_1(N)$ is 
a Riemann surface and can be compactified into a compact Riemann surface   $X_1(N)$ 
(see \cite[Chapter 1.5]{diamond-shurman}). 
The genus of $X_1(N)$ can be given as an explicit function of $N$ (cf. 
\cite[Ex.3.1.6]{diamond-shurman}). For 
$n_0 =1728$, $X_1(n_0)$ has genus greater than $2$.  
\par
Let $B$ be a smooth projective complex curve of genus $g$ and let $S \subset B$ be a finite subset of cardinality $s \geq 1$. 
Two elliptic surfaces 
$\mathcal{E} \to B$ and $\mathcal{E}' \to B$ are said to be \emph{equivalent} if there exists a $B$-isomorphism 
$\mu \colon \mathcal{E} \to \mathcal{E}'$. We have:

\begin{theorem}
[Uniform Shafarevich's theorem for elliptic surfaces]  
\label{t:uniform-shafarevich-prob}
There exists a function $A \colon \N^2 \to \N$ such that 
the set of equivalence classes of nonisotrivial minimal elliptic surfaces 
$\mathcal{E} \to B$ with good reduction away from $S$  has no more than  
$A(g,s)$ elements. 
 \end{theorem}

 \begin{proof}
 $U\coloneqq B \backslash S$ is an open affine curve. 
We claim that for every elliptic scheme $\mathcal{E} \to U$, there is a finite \' etale cover $U' \to U$ 
of uniformly bounded degree (in terms of $g,s$) over which $\mathcal{E}$ has a section of order $n_0$. 
Indeed, since $\pi_1(U)$ has only finitely many quotients of a given size 
(which corresponds to the degree of $U' \to U$), there are only finitely many possibilities for $U'$ up to $U$-isomorphism. 
The number of such $U'$ is bounded in terms of the free group $\pi_1(U)\simeq F_{2g-1+s}$ and thus uniformly bounded in terms of $g, s$. 
Each connected component $U_i'$ of $U'$ has a moduli map to $Y_1(n_0)$ 
which is dominant if the $j$-invariant $j(E)$ is not constant. 
As $X_1(n_0)$ has genus at least 2, the de Franchis theorem 
(cf. \cite[Theorem XXI.8.27]{Arbarello-II}) tells us that there are only finitely many such moduli maps. 
Therefore, there exists a finite \' etale extension $V \to U$ 
of uniformly bounded degree over which every nonisotrivial elliptic scheme $\mathcal{E} \to U$ has a section of order $n_0$ as claimed. 
Again, an effective version of the de Franchis theorem (cf. \cite{alzati-pirola-90}) 
implies that  the number of nonconstant maps $V \to Y_1(n_0)$ is uniformly bounded. 
The proof is completed. 
 \end{proof}

\subsection{Simple cyclic covers}
Let $k$ be an algebraically closed field of characteristic $0$. 
 
\begin{proposition} 
[Simple cyclic covers]
\label{p:cyclic-covers}
Let $m\in \N$. Let $X/k$ be an irreducible projective smooth variety. 
Let $D$ be an effective divisor of $X$ such that $D \sim L^{\otimes m}$ for some line bundle $L$ of $X$. 
There exists a unique finite cyclic cover of degree $m$ of irreducible projective $k$-varieties 
$f_{D, L, m} \colon X' \to X$ 
 such that $f$ is totally ramified and ramified only above $D$ and every point of     
$f_{D,L,m}^{-1} (X \setminus D_{\text{sing}})$ is a regular point of $X'$. 
\end{proposition}

\begin{proof}
See for example \cite{pardini-91}, in particular \cite[Proposition 3.1]{pardini-91}. 
\end{proof}

Moreover, it is clear that construction of simple cyclic covers is functorial:  

\begin{proposition}
\label{p:cyclic-covers-fonctorial}
Let $X, Y$ be smooth irreducible projective varieties. 
Let $g \colon Y \to X$ be a morphism such that 
$g^*D$ is well defined where $D$ is an effective divisor on $X$. 
Let $m \in \N$ and $L$ a line bundle such that $L^{\otimes m}\sim \OO_X(D)$. 
We have a cartesian square: 

\[
\label{d-parshin-integral-point-diagram-etale}
\begin{tikzcd}
 Y'  \arrow[r,"f_{g^*D, g^*L, m}"]  \arrow[d]  & [3em] Y \arrow[d,"g"]  \\ 
 X' \arrow[r, "f_{D,L,m}"]   & X.  
\end{tikzcd}
\]
\end{proposition}
 
\subsection{Geometry of elliptic surfaces} 

\begin{theorem} 
\label{t:translation}
Let $f\colon X\to B$ be a minimal elliptic surface with a section $(O)$ and let $E/K$ be the associated elliptic curve. Then we have the followings:
\begin{enumerate} [\rm (1)]
\item
For each $P \in E(K)=X(B)$, 
the rational  maps $\tau_P\colon X \to X$ induced by the translation by $P$ on each regular fiber extends, by minimality of $X$, to an $B$-automorphism.
\item
 We have a homomorphism of groups 
 $E(K)  \to \text{Aut}_B(X)$, $P  \mapsto \tau_P$. 
In particular, $\tau_{-P}$ is the inverse of $\tau_P$ for any $P\in E(K)$.
\item
We have a canonical isomorphism of groups $\text{Isom}_K(E)\simeq \Aut_B(X)$.  
\item
The isomorphism group of $E$ over $K$ is given by 
$\text{Isom}_K(E)\simeq E(K) \rtimes \Aut_K(E)$  
where $\#\Aut_{\overline{K}}(E)=2,4,6$ according to $j_E\neq 0,1728, j_E=1728$ or $j_E=0$. 
\end{enumerate}
\end{theorem}

\begin{proof}
See for example \cite[Theorem III.9.1]{silverman-ATACE} for (1), (2), (3), \cite[Theorem 0]{parshin-68} for (3),  
and \cite[AEC III.10.1]{silverman-ACE} for (4). 
\end{proof}

\bibliographystyle{siam}

\end{document}